\documentclass[11pt, a4paper]{article}
\usepackage{amssymb,amsmath,amsthm}
\usepackage{fancyhdr}
\usepackage[british]{babel}
\usepackage{enumitem}
\usepackage[dvipsnames]{xcolor}

\usepackage{hyperref}
\usepackage[margin=2.5cm]{geometry}
\usepackage{bbm}
\newcounter{propcounter}

\usepackage{todonotes}

\usepackage{graphicx}

\usepackage{amsmath,bm}
\usepackage{amssymb}
\usepackage{amsfonts}
\usepackage{verbatim}
\usepackage{scalefnt}
\usepackage{multirow}
\usepackage{mathtools}

\usepackage{ifthen,latexsym,fixmath}

\def\pseudop{{\star}}

\usepackage{pstricks,tikz}

\usepackage{fancyhdr}
\usepackage{hyperref}

\newif\ifnotesw\noteswtrue

\newcommand\innerp[1]{\langle #1 \rangle}

\newenvironment{poc}{\begin{proof}[Proof of claim]}{\end{proof}}

\newcommand{\V}[1]{\mathbold{#1}}

\newcommand{\COMMENT}[1]{}
\renewcommand{\COMMENT}{\footnote}

\newtheorem{theorem}{Theorem}[section]
\newtheorem{lemma}[theorem]{Lemma}
\newtheorem{observation}[theorem]{Observation}
\newtheorem{proposition}[theorem]{Proposition}
\newtheorem{corollary}[theorem]{Corollary}
\newtheorem{claim}[theorem]{Claim}

\newtheorem{conjx}{Conjecture}
\newtheorem{probx}[conjx]{Problem}

\theoremstyle{definition}
\newtheorem{remark}[theorem]{Remark}
\newtheorem{definition}[theorem]{Definition}

\def\al{{\alpha}}
\def\de{{\delta}}
\def\eps{{\varepsilon}}

\newcommand{\iminnerp}[1]{{\rm{Im}}\langle #1 \rangle}
\newcommand{\reinnerp}[1]{{\rm{Re}}\langle #1 \rangle}

\newcommand{\im}[1]{{\rm{Im}}(#1)}

\newcommand{\diam}[1]{{\rm{diam}}(#1)}

\newcommand{\cB}{\mathcal{B}}
\newcommand{\cG}{\mathcal{G}}

\newcommand{\floor}[1]{\lfloor{#1}\rfloor}

\newcommand{\sS}{\mathsf{S}}
\newcommand{\sN}{\mathbb{N}}
\newcommand{\sC}{\mathbb{C}}
\newcommand{\sR}{\mathbb{R}}
\newcommand{\sZ}{\mathbb{Z}}

\newcommand{\rt}{{\mathsf{RT}}}

\newcommand*{\bigabs}[1]{\bigl|#1\bigr|}

\title{Geometric constructions for Ramsey-Tur\'an theory}

\author{
	Hong Liu
	\thanks{Extremal Combinatorics and Probability Group (ECOPRO), Institute for Basic Science (IBS), South Korea, and Mathematics Institute and DIMAP, University of Warwick, UK. Email: \texttt{hongliu@ibs.re.kr}. Supported by the Institute for Basic Science (IBS-R029-C4) and  the UK Research and Innovation Future Leaders Fellowship MR/S016325/1.}
	\and
	Christian Reiher
	\thanks{Fachbereich Mathematik, Universit\"at Hamburg, Germany. Email: {\tt Christian.Reiher@uni-hamburg.de}.}
	\and
	Maryam Sharifzadeh 
	\thanks{Department of Mathematics and Mathematical Statistics, Ume\r{a} University, Sweden. Email: {\tt maryam.sharifzadeh@umu.se}.}
	\and 
	Katherine Staden
	\thanks{School of Mathematics and Statistics, Open University, UK, and Mathematical Institute, University of Oxford, UK. Email: {\tt katherine.staden@open.ac.uk}.}
}

\begin{document}
	
\maketitle

\begin{abstract}
	Combining two classical notions in extremal combinatorics, the study of Ramsey-Tur\'an theory seeks to determine, for integers $m\le n$ and $p \leq q$, the number $\rt_p(n,K_q,m)$, which is the maximum size of an $n$-vertex $K_q$-free graph in which every set of at least $m$ vertices contains a $K_p$. 
	
	Two major open problems in this area from the 80s ask: (1) whether the asymptotic extremal structure for the general case exhibits certain periodic behaviour, resembling that of the special case when $p=2$; (2) constructing analogues of Bollob\'as-Erd\H{o}s graphs with densities other than $1/2$.
	
	We refute the first conjecture by witnessing asymptotic extremal structures that are drastically different from the $p=2$ case, and address the second problem by constructing Bollob\'as-Erd\H{o}s-type graphs using high dimensional complex spheres with \emph{all rational} densities. Some matching upper bounds are also provided.
\end{abstract}




\section{Introduction}
Ramsey graphs for cliques are believed to be random-like; while on the other hand, the Tur\'an graphs from extremal graph theory are highly structured. Initiated in 1969 by S\'os, and later generalised by Erd\H{o}s, Hajnal, S\'os and Szemer\'edi~\cite{EHSSz}, \emph{Ramsey-Tur\'an theory} combines flavours of graph Ramsey and Tur\'an problems. The \emph{Ramsey-Tur\'an number} $\rt_p(n,K_q,m)$ is the maximum number of edges in an $n$-vertex $K_q$-free graph $G$ with $\alpha_p(G)\le m$, where $\al_p(G)=\max\{|U|:U\subseteq V(G)\text{ and }G[U]\text{ is }K_p\text{-free}\}$ is the \emph{$p$-independence number} of $G$. Notice that when $p=2$ and $m=n$, we recover the Tur\'an number of $K_q$; and as we consider large graphs, that is $n\rightarrow\infty$, by Ramsey's theorem, $m$ should be taken as a function of $n$.

 Aside from its close connection to Ramsey theory, e.g.~the seminal result of Ajtai, Koml\'os and Szemer\'edi~\cite{AKSz} on the independence number of triangle-free graphs with given size, results in Ramsey-Tur\'an theory have been applied, for instance, to construct dense infinite Sidon sets~\cite{AKSz-Sidon} in additive number theory, and to refute Heilbronn's conjecture~\cite{KPSz} in discrete geometry. For more details, we refer the reader to the comprehensive survey of Simonovits and S\'os~\cite{SimSos}.

In this paper, we consider the most classical setting $\rt_p(n,K_q,o(n))$, when the independence number is sublinear.

\subsection{Background}
The classical setting of sublinear independence number is defined as follows. Let $\varrho_p(q)$ be the \emph{Ramsey-Tur\'an density}: 
$$
\varrho_p(q) := \lim_{\eps\rightarrow 0}\lim_{n\rightarrow \infty}\frac{\rt_p(n,K_q,\eps n)}{\binom{n}{2}}.
$$
The existence of the limit was shown by Erd\H{o}s, Hajnal, Simonovits, S\'os and Szemer\'edi~\cite{EHSSSz}. Then $\rt_p(n,K_q,o(n)):=\varrho_p(q){n\choose 2}+o(n^2)$.

For $p=2$, the problem is now well-understood. First, in
1970, Erd\H{o}s and S\'os~\cite{ES} proved that $\varrho_2(2t+1) = \frac{t-1}{t}$, for all $t \geq 1$.
The case of even cliques turned out to be much harder. Applying a proto-regularity lemma, Szemer\'edi~\cite{SzK4} showed in 1973 that $\varrho_2(4) \leq \frac{1}{4}$. It was suspected by many that perhaps dense $K_4$-free graphs with sublinear independence number do not exist, i.e.~$\varrho_2(4)=0$. Then, surprisingly, a matching lower bound was given by Bollob\'as and Erd\H{o}s in 1976; their ingenious construction -- now called the \emph{Bollob\'as-Erd\H{o}s graph} -- was based on high dimensional spheres. 
Eventually in 1983, Erd\H{o}s, Hajnal, S\'os and Szemer\'edi~\cite{EHSSz} completed the $p=2$ case, proving that $\varrho_2(2t)=\frac{3t-5}{3t-2}$ for all $t \geq 2$. Furthermore, they showed that $\varrho_2(q)$ exhibits the following \emph{periodical} behaviour:
\begin{itemize}
	\item[$(\pseudop)$] Let $G$ be an asymptotic extremal graph for $\varrho_2(2t+r+2)$ with $r\in\{0,1\}$. Then the vertex set $V(G)$ can be partitioned into $V_0\cup V_1\cup \ldots \cup V_t$, such that
	\begin{itemize}
		\item each $G[V_i]$ has $o(1)$ edge-density;
		
		\item $G[V_0,V_1]$ has density $\frac{r+1}{2}-o(1)$; 
		
		\item every other $G[V_i,V_j]$ has density $1-o(1)$.
	\end{itemize}
\end{itemize}
In other words, the asymptotic extremal structure depends on the residue of $q$ modulo $p$ and evolves as follows: the density of the pair $G[V_0,V_1]$ increases as $r$, the residue of $q\!\mod p=2$, increases; and whenever $q$ increases by $p=2$, a new part is added and joined completely to previous parts.


The general problem $\varrho_p(q)$ for $p>2$ has been notoriously difficult and remained largely open. Indeed, apart from the trivial case $\varrho_p(p+1)=0$, the next simplest case $\varrho_3(5)$ remained open before this work. Quoting Erd\H{o}s, Hajnal, Simonovits, S\'os and Szemer\'edi~\cite{EHSSSz}, \emph{``One of the most intriguing problems is to determine the values and some asymptotically extremal graphs for $\rt_3(n,K_5,o(n))$ and $\rt_3(n,K_6,o(n))$. Unfortunately, this task seems to be too difficult.''} Despite this, in the same paper, they proposed the following bold conjecture, predicting that similarly to $\varrho_2(q)$ in~($\pseudop$), the general problem $\varrho_p(q)$ also has similar periodic asymptotic extremal structures. In particular, the value of $\varrho_p(q)$ depends on the residue of $q\!\mod p$ (see Figure~\ref{fig:periodic}).

\begin{conjx}[\cite{EHSSSz}, Conjecture~2.9]\label{conj}
	The asymptotic extremal graphs $G$ for $\varrho_p(q)$ have the following structure.
	Let $q=pt+r+2$ where $t\in\mathbb{N}$ and $0 \leq r < p$. Then there is a partition $V(G)= V_0\cup V_1 \cup \ldots \cup V_t$ such that
	\begin{itemize}
		\item $e(G[V_i])=o(n^2)$ for all $0\le i \le t$;
		
		\item $d_G(V_0,V_1) = \frac{r+1}{p}-o(1)$, and degrees in $G[V_0,V_1]$ differ by $o(n)$;
		
		\item $d_G(V_i,V_j) =1-o(1)$ for all pairs $\lbrace i,j\rbrace \neq \lbrace 0,1\rbrace$.
	\end{itemize}
	In particular, 
	\begin{equation}\label{eq:conjRTdensity}
		\varrho_p(q)=\varrho_p^*(q):=\frac{(t-1)(2p-r-1)+r+1}{t(2p-r-1)+r+1}.
	\end{equation}
\end{conjx}

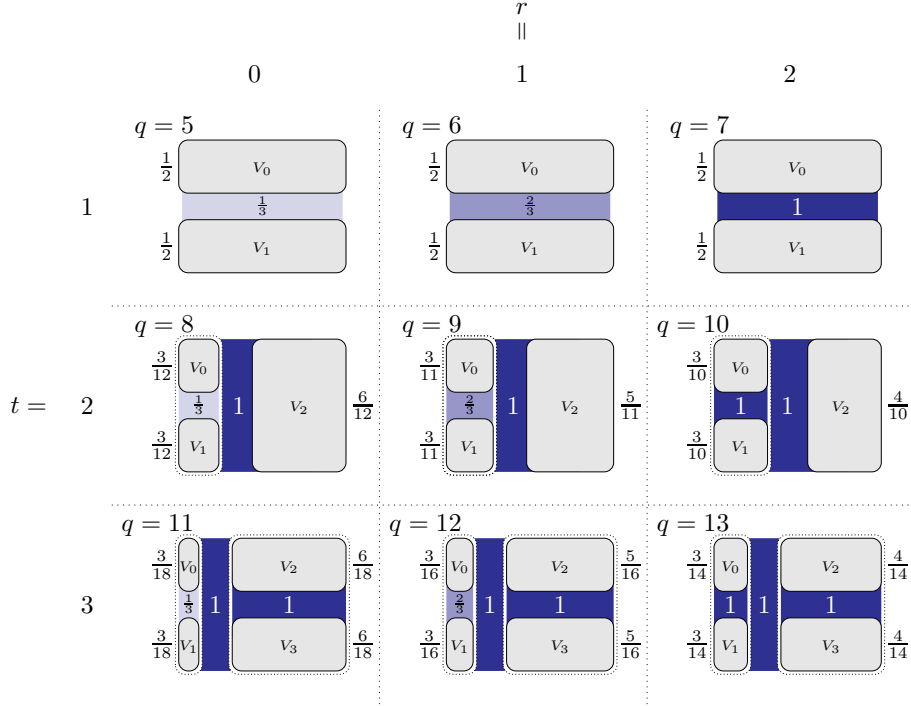
\begin{figure}[h]
	\centering
	\scalebox{0.88}{	
		\begin{tikzpicture}	
		\begin{scope}[xshift=-3.5cm,yshift=0.5cm]
		\draw[line width=2.4cm,color=Blue, opacity=0.2] (-0.25,0.5) -- (-0.25,1.5);
		\draw[rounded corners, fill=gray!20] (-1.5,1.2) rectangle (1,2);
		\draw[rounded corners, fill=gray!20] (-1.5,0) rectangle (1,0.8);
		
		\node[label=above:\tiny$V_0$] at (-0.25,1.2) {};
		\node[label=above:\tiny$V_1$] at (-0.25,0.0) {};
		
		\node[label=left:\small$\frac{1}{2}$] at (-1.3,1.6) {};
		\node[label=left:\small$\frac{1}{2}$] at (-1.3,0.4) {};
		
		\node[label=above:\tiny$\frac{1}{3}$] at (-0.25,0.6) {};
		
		\node[label=left:{$q=5$}] at (-1.,2.2) {};
		\end{scope}
		
		\begin{scope}[xshift=0.5cm,yshift=0.5cm]
		\draw[line width=2.4cm,color=Blue, opacity=0.5] (-0.25,0.5) -- (-0.25,1.5);
		\draw[rounded corners, fill=gray!20] (-1.5,1.2) rectangle (1,2);
		\draw[rounded corners, fill=gray!20] (-1.5,0) rectangle (1,0.8);
		
		\node[label=above:\tiny$V_0$] at (-0.25,1.2) {};
		\node[label=above:\tiny$V_1$] at (-0.25,0.0) {};
		
		\node[label=left:\small$\frac{1}{2}$] at (-1.3,1.6) {};
		\node[label=left:\small$\frac{1}{2}$] at (-1.3,0.4) {};
		
		\node[label=above:\tiny$\frac{2}{3}$] at (-0.25,0.6) {};
		
		\node[label=left:{$q=6$}] at (-1.,2.2) {};
		\end{scope}
		
		\begin{scope}[xshift=4.5cm,yshift=0.5cm]
		\draw[line width=2.4cm,color=Blue, opacity=1.0] (-0.25,0.5) -- (-0.25,1.5);
		\draw[rounded corners, fill=gray!20] (-1.5,1.2) rectangle (1,2);
		\draw[rounded corners, fill=gray!20] (-1.5,0) rectangle (1,0.8);
		
		\node[label=above:\tiny$V_0$] at (-0.25,1.2) {};
		\node[label=above:\tiny$V_1$] at (-0.25,0.0) {};
		
		\node[label=left:\small$\frac{1}{2}$] at (-1.3,1.6) {};
		\node[label=left:\small$\frac{1}{2}$] at (-1.3,0.4) {};
		
		\node[label=above:\textcolor{white}{$1$}] at (-0.25,0.6) {};
		
		\node[label=left:{$q=7$}] at (-1.,2.2) {};
		\end{scope}

		
		\begin{scope}[xshift=-3.5cm,yshift=-2.5cm]
		\draw[line width=2cm,color=Blue, opacity=1.0] (-1,1) -- (0.5,1);
		\draw[rounded corners, densely dotted,fill=white] (-1.55,-.05) rectangle (-0.85,2.05);
		\draw[line width=0.6cm,color=Blue, opacity=0.2] (-1.2,0.5) -- (-1.2,1.5);
		\draw[rounded corners, fill=gray!20] (-1.5,0) rectangle (-0.9,0.8);
		\draw[rounded corners, fill=gray!20] (-1.5,1.2) rectangle (-0.9,2);
		\draw[rounded corners, fill=gray!20] (-0.4,0) rectangle (1,2);
		
		\node[label=above:\tiny$V_0$] at (-1.2,1.2) {};
		\node[label=above:\tiny$V_1$] at (-1.2,0.0) {};
		\node[label=above:\tiny$V_2$] at (0.3,0.6) {};
		
		\node[label=left:\small$\frac{3}{12}$] at (-1.3,1.6) {};
		\node[label=left:\small$\frac{3}{12}$] at (-1.3,0.4) {};
		\node[label=right:\small$\frac{6}{12}$] at (0.8,1) {};
		
		\node[label=above:\tiny$\frac{1}{3}$] at (-1.2,0.6) {};
		\node[label=above:\textcolor{white}{$1$}] at (-0.6,0.6) {};
		
		\node[label=left:{$q=8$}] at (-1.,2.2) {};
		\end{scope}
		
		\begin{scope}[xshift=0.5cm,yshift=-2.5cm]
		\draw[line width=2cm,color=Blue, opacity=1.0] (-1,1) -- (0.5,1);
		\draw[rounded corners, densely dotted, fill=white] (-1.55,-0.05) rectangle (-0.75,2.05);
		\draw[line width=0.7cm,color=Blue, opacity=0.5] (-1.15,0.5) -- (-1.15,1.5);
		\draw[rounded corners, fill=gray!20] (-1.5,0) rectangle (-0.8,0.8);
		\draw[rounded corners, fill=gray!20] (-1.5,1.2) rectangle (-0.8,2);
		\draw[rounded corners, densely dotted] (-1.55,-0.05) rectangle (-0.75,2.05);
		\draw[rounded corners, fill=gray!20] (-0.3,0) rectangle (1,2);
		
		\node[label=above:\tiny$V_0$] at (-1.15,1.2) {};
		\node[label=above:\tiny$V_1$] at (-1.15,0.0) {};
		\node[label=above:\tiny$V_2$] at (0.35,0.6) {};
		
		\node[label=left:\small$\frac{3}{11}$] at (-1.3,1.6) {};
		\node[label=left:\small$\frac{3}{11}$] at (-1.3,0.4) {};
		\node[label=right:\small$\frac{5}{11}$] at (0.8,1) {};
		
		\node[label=above:\tiny$\frac{2}{3}$] at (-1.15,0.6) {};
		\node[label=above:\textcolor{white}{$1$}] at (-0.5,0.6) {};
		
		\node[label=left:{$q=9$}] at (-1.,2.2) {};
		\end{scope}
		
		\begin{scope}[xshift=4.5cm,yshift=-2.5cm]
		\draw[line width=2cm,color=Blue, opacity=1.0] (-1,1) -- (0.5,1);
		\draw[rounded corners, densely dotted, fill=white] (-1.55,-0.05) rectangle (-0.65,2.05);
		\draw[line width=0.8cm,color=Blue, opacity=1.0] (-1.1,0.5) -- (-1.1,1.5);
		\draw[rounded corners, fill=gray!20] (-1.5,0) rectangle (-0.7,0.8);
		\draw[rounded corners, fill=gray!20] (-1.5,1.2) rectangle (-0.7,2);
		\draw[rounded corners, fill=gray!20] (-0.1,0) rectangle (1,2);
		
		\node[label=above:\tiny$V_0$] at (-1.1,1.2) {};
		\node[label=above:\tiny$V_1$] at (-1.1,0.0) {};
		\node[label=above:\tiny$V_2$] at (0.4,0.6) {};
		
		\node[label=left:\small$\frac{3}{10}$] at (-1.3,1.6) {};
		\node[label=left:\small$\frac{3}{10}$] at (-1.3,0.4) {};
		\node[label=right:\small$\frac{4}{10}$] at (0.8,1) {};
		
		\node[label=above:\textcolor{white}{$1$}] at (-1.1,0.6) {};
		\node[label=above:\textcolor{white}{$1$}] at (-0.4,0.6) {};
		
		\node[label=left:{$q=10$}] at (-1.,2.2) {};
		\end{scope}
		

		\begin{scope}[xshift=-3.5cm,yshift=-5.5cm]
		\draw[line width=2cm,color=Blue, opacity=1.0] (-1.5,1) -- (0.5,1);
		\draw[rounded corners, densely dotted,fill=white] (-1.55,-0.05) rectangle (-1.15,2.05);
		\draw[line width=0.3cm,color=Blue, opacity=0.2] (-1.35,0.5) -- (-1.35,1.5);
		\draw[rounded corners, fill=gray!20] (-1.5,0) rectangle (-1.2,0.8);
		\draw[rounded corners, fill=gray!20] (-1.5,1.2) rectangle (-1.2,2);
		\draw[rounded corners, densely dotted,fill=white] (-0.75,-0.05) rectangle (1.05,2.05);
		\draw[line width=1.7cm,color=Blue, opacity=1.0] (0.15,0.5) -- (0.15,1.5);
		\draw[rounded corners, fill=gray!20] (-0.7,0) rectangle (1,0.8);
		\draw[rounded corners, fill=gray!20] (-0.7,1.2) rectangle (1,2);
		
		\node[label=above:\tiny$V_0$] at (-1.35,1.2) {};
		\node[label=above:\tiny$V_1$] at (-1.35,0.0) {};
		\node[label=above:\tiny$V_2$] at (0.15,1.2) {};
		\node[label=above:\tiny$V_3$] at (0.15,0.0) {};
		
		\node[label=left:\small$\frac{3}{18}$] at (-1.3,1.6) {};
		\node[label=left:\small$\frac{3}{18}$] at (-1.3,0.4) {};
		\node[label=right:\small$\frac{6}{18}$] at (0.8,1.6) {};
		\node[label=right:\small$\frac{6}{18}$] at (0.8,0.4) {};
		
		\node[label=above:\tiny$\frac{1}{3}$] at (-1.35,0.6) {};
		\node[label=above:\textcolor{white}{$1$}] at (-0.95,0.6) {};
		\node[label=above:\textcolor{white}{$1$}] at (0.15,0.6) {};
		
		\node[label=left:{$q=11$}] at (-1.,2.2) {};
		\end{scope}
		
		\begin{scope}[xshift=0.5cm,yshift=-5.5cm]
		\draw[line width=2cm,color=Blue, opacity=1.0] (-1.5,1) -- (0.5,1);
		\draw[rounded corners, densely dotted,fill=white] (-1.55,-0.05) rectangle (-1.05,2.05);
		\draw[line width=0.4cm,color=Blue, opacity=0.5] (-1.3,0.5) -- (-1.3,1.5);
		\draw[rounded corners, fill=gray!20] (-1.5,0) rectangle (-1.1,0.8);
		\draw[rounded corners, fill=gray!20] (-1.5,1.2) rectangle (-1.1,2);
		\draw[rounded corners, densely dotted,fill=white] (-0.65,-0.05) rectangle (1.05,2.05);
		\draw[line width=1.6cm,color=Blue, opacity=1.0] (0.2,0.5) -- (0.2,1.5);
		\draw[rounded corners, fill=gray!20] (-0.6,0) rectangle (1,0.8);
		\draw[rounded corners, fill=gray!20] (-0.6,1.2) rectangle (1,2);
		
		\node[label=above:\tiny$V_0$] at (-1.3,1.2) {};
		\node[label=above:\tiny$V_1$] at (-1.3,0.0) {};
		\node[label=above:\tiny$V_2$] at (0.2,1.2) {};
		\node[label=above:\tiny$V_3$] at (0.2,0.0) {};
		
		\node[label=left:\small$\frac{3}{16}$] at (-1.3,1.6) {};
		\node[label=left:\small$\frac{3}{16}$] at (-1.3,0.4) {};
		\node[label=right:\small$\frac{5}{16}$] at (0.8,1.6) {};
		\node[label=right:\small$\frac{5}{16}$] at (0.8,0.4) {};
		
		\node[label=above:\tiny$\frac{2}{3}$] at (-1.3,0.6) {};
		\node[label=above:\textcolor{white}{$1$}] at (-0.85,0.6) {};
		\node[label=above:\textcolor{white}{$1$}] at (0.2,0.6) {};
		
		\node[label=left:{$q=12$}] at (-1.,2.2) {};
		\end{scope}
		
		\begin{scope}[xshift=4.5cm,yshift=-5.5cm]
		\draw[line width=2cm,color=Blue, opacity=1.0] (-1.5,1) -- (0.5,1);
		\draw[rounded corners, densely dotted,fill=white] (-1.55,-0.05) rectangle (-0.95,2.05);
		\draw[line width=0.5cm,color=Blue, opacity=1.0] (-1.25,0.5) -- (-1.25,1.5);
		\draw[rounded corners, fill=gray!20] (-1.5,0) rectangle (-1.0,0.8);
		\draw[rounded corners, fill=gray!20] (-1.5,1.2) rectangle (-1.0,2);
		\draw[rounded corners, densely dotted,fill=white] (-0.55,-0.05) rectangle (1.05,2.05);
		\draw[line width=1.5cm,color=Blue, opacity=1.0] (0.25,0.5) -- (0.25,1.5);
		\draw[rounded corners, fill=gray!20] (-0.5,0) rectangle (1,0.8);
		\draw[rounded corners, fill=gray!20] (-0.5,1.2) rectangle (1,2);

		\node[label=above:\tiny$V_0$] at (-1.25,1.2) {};
		\node[label=above:\tiny$V_1$] at (-1.25,0.0) {};
		\node[label=above:\tiny$V_2$] at (0.25,1.2) {};
		\node[label=above:\tiny$V_3$] at (0.25,0.0) {};
		
		\node[label=left:\small$\frac{3}{14}$] at (-1.3,1.6) {};
		\node[label=left:\small$\frac{3}{14}$] at (-1.3,0.4) {};
		\node[label=right:\small$\frac{4}{14}$] at (0.8,1.6) {};
		\node[label=right:\small$\frac{4}{14}$] at (0.8,0.4) {};
		
		\node[label=above:\textcolor{white}{$1$}] at (-1.25,0.6) {};
		\node[label=above:\textcolor{white}{$1$}] at (-0.75,0.6) {};
		\node[label=above:\textcolor{white}{$1$}] at (0.25,0.6) {};
		
		\node[label=left:{$q=13$}] at (-1.,2.2) {};
		\end{scope}

		\draw[black,dotted] (-6,0) -- (6,0);
		\draw[black,dotted] (-6,-3) -- (6,-3);
		
		\draw[black,dotted] (-2,-6) -- (-2,3);
		\draw[black,dotted] (2,-6) -- (2,3);
		
		\newcommand{\verteq}{\rotatebox{90}{$\,=$}}
		\newcommand{\equalto}[2]{\underset{\scriptstyle\overset{\mkern4mu\verteq}{#2}}{#1}}

		\node[label=left:{$\equalto{r}{}$}] at (0.5,4.2) {};
		\node[label=left:{$0$}] at (-3.5,3.5) {};
		\node[label=left:{$1$}] at (0.5,3.5) {};
		\node[label=left:{$2$}] at (4.5,3.5) {};
		
		\node[label=left:{$t=$}] at (-6.7,-1.5) {};
		\node[label=left:{$1$}] at (-6.0,1.5) {};
		\node[label=left:{$2$}] at (-6.0,-1.5) {};
		\node[label=left:{$3$}] at (-6.0,-4.5) {};
		\end{tikzpicture}
	}
	\caption{An illustration for $p=3$.\label{fig:periodic}}
\end{figure}

In the final assertion, $\varrho_p^*(q)$ is obtained by
optimising the sizes of the vertex classes in the graph predicted by the conjecture. Towards this major conjecture, in~\cite{EHSSSz}, an upper bound of $\varrho_p(q) \leq \frac{q-1-p}{q-1}$ was proven, which is optimal when $q\equiv 1\!\mod p$, verifying~\eqref{eq:conjRTdensity} for this special case; and for sporadic cases when $q=p+\ell$, $\ell\le\min\{5,p\}$, it was shown that $\varrho_p(p+\ell) \leq \varrho_p^*(p+\ell) = \frac{\ell-1}{2p}$. Conjecture~\ref{conj} remains wide open for $q\not\equiv 1\!\mod p$.

As in many other extremal problems, when determining the Ramsey-Tur\'an density $\varrho_p(q)$, obtaining explicit constructions for the lower bound is the most challenging aspect. In this direction, even the simplest subproblem of determining whether $\varrho_3(5)>0$ was only confirmed in 2011 by a breakthrough of Balogh and Lenz~\cite{BL-1} using an elegant construction. The best general lower bound~\cite{BL-2} when $\ell\le p$ is $\varrho_p(p+\ell)\ge \frac{1}{2^{k+1}}$, where $\lceil \frac{p}{2^k}\rceil<\ell$, which provides the state of the art for $\varrho_3(5)$:
\begin{equation*}
	\frac{1}{8}\le \varrho_3(5)\le \frac{1}{6}.
\end{equation*}

It was stated in the work of Erd\H{o}s, Hajnal, S\'os and Szemer\'edi~\cite{EHSSz} that, for $\varrho_3(5)$, \emph{``an analogue of the Bollob\'as-Erd\H{o}s graph would be needed which we think will be extremely hard to find.''} This motivates another main open problem in this area:
\begin{probx}[\cite{BL-1,EHSSz}]\label{problem-BE}
	Construct an analogue of the Bollob\'as-Erd\H{o}s graph with density other than $\frac{1}{2}$.
\end{probx}

The only progress towards Problem~\ref{problem-BE} was the aforementioned results of Balogh and Lenz~\cite{BL-1,BL-2}, taking a certain product construction utilising Bollobas-Erd\H{o}s graphs to get variations with densities equal to powers of $1/2$.
Several other problems were also raised whose solution would make progress on Conjecture~\ref{conj}; we refer the reader to~\cite{BL-2,EHSSSz}.

\medskip

In this paper, we address all of these problems, revealing some unexpected phenomena of Ramsey-Tur\'an graphs.

\subsection{Complex Bollob\'as-Erd\H{o}s graphs with rational densities}
Our first main result answers Problem~\ref{problem-BE}.
Inspired by the Bollob\'as-Erd\H{o}s graph, we use isoperimetry and concentration of measure on the high dimensional \emph{complex} sphere to achieve all rational densities.

\begin{theorem}[Complex Bollob\'as-Erd\H{o}s graph]\label{GBE}
	Let $p,\ell$ be integers with $1 \leq \ell <p $. Then for all sufficiently large $n$,  there exists a graph $G$ with vertex partition $W \cup Z$, where $|W|=|Z|=n$, such that $\alpha_p(G) =o(n)$, $e(G[W]), e(G[Z])=o(n^2)$, and $e_G(W,Z) = (\ell/p-o(1))n^2$. If additionally $\ell\le p/2$, then $G$ is $K_{p+\ell+1}$-free, and consequently,
	$$\varrho_p(p+\ell+1) \geq \frac{\ell}{2p} = \varrho_p^*(p+\ell+1).$$
\end{theorem}

An immediate corollary of this result is that there is a construction as in Conjecture~\ref{conj} for just over half of all cases: we let $G[V_0 \cup V_1]$ be the graph in Theorem~\ref{GBE}.

\begin{corollary}\label{cor}
	Let $q = pt+\ell+1$. Then for all $0 \leq \ell \leq p/2$, 
	$$\varrho_p(q) \geq \varrho_p^*(q).$$
\end{corollary}

This in particular determines, after about 40 years, that $\varrho_3(5)=\frac{1}{6}$.

\subsection{Non-periodicity of Ramsey-Tur\'an graphs}
Our second main result, much to our own surprise, disproves Conjecture~\ref{conj} for infinitely many cases, all having densities strictly larger than the predicted $\varrho_p^*(q)$. 

For instance, Conjecture~\ref{conj} claims $\varrho_m(m+11)=\varrho_m^*(m+11)=\frac{5}{m}$
for every $m \geq 10$, with equality being achieved by an almost bipartite graph having density $\frac{10}{m}$ between the vertex classes. Our results show, however, that at least if $m=2^\ell$ is a power of $2$ with exponent $\ell \geq 9$, then $\varrho_m(m+11)=\frac{6}{m}$, where almost $4$-partite graphs with density $\frac{8}{m}$ between their vertex classes are extremal. The lower bound can be seen by plugging $p=3$ and $q=4$ into the statement that follows.

\begin{theorem}\label{thm-2tri-exact}
	Let $\ell,p,q\in\mathbb{N}$ with $q$ even, $\ell \geq p(q-1)$, $p^{\star}:=2^\ell$ and $q^{\star}:=2^\ell+2^p+q-1$. Then for all sufficiently large $n$, there exists an $n$-vertex $K_{q^{\star}}$-free graph $G$ with $\alpha_{p^{\star}}(G)=o(n)$ and an  equipartition $V(G)=V_1\cup\ldots\cup V_q$ such that
	\begin{itemize}
		\item $e(G[V_i])=o(n^2)$ for each $i\in[q]$;
		
		\item $d_G(V_i,V_j)=\frac{1}{2^{\ell-p}}-o(1)$ for all $ij\in{[q]\choose 2}$.
	\end{itemize}
	In particular, 
	\begin{equation}\label{eq:multi-lower}
		\varrho_{p^{\star}}(q^{\star})\ge\frac{1}{2^{\ell-p}}\left(1-\frac{1}{q}\right),
	\end{equation}
	where equality holds when $q(q-2)\le 2^p\le q^2$; and whenever $q > 2$, $$\varrho_{p^{\star}}(q^{\star}) > \varrho_{p^{\star}}^*(q^{\star}).$$
\end{theorem}

It is worth noting that Theorem~\ref{thm-2tri-exact} in fact refutes Conjecture~\ref{conj} in a strong sense. It reveals that the asymptotic extremal structure for $\varrho_p(q)$ is much more intricate. Indeed, the graph predicted in Conjecture~\ref{conj} remains almost bipartite when $q\le 2p+1$, in this range the cross density increases by $1/p$ when $q$ increases by one; and after this point, for each increment of $q$ by $p$, an additional part is added and joined completely to previous parts. Theorem~\ref{thm-2tri-exact} shows that already when $q\le 2p+1$, the asymptotic extremal structure could be almost $t$-partite for infinitely many choices of $t$.

\subsection{Some matching upper bounds}
Our remaining results concern upper bounds. Combined with our constructions in Theorem~\ref{GBE}, they show that $\varrho_p^*(pt+2)$ in Conjecture~\ref{conj} is the correct value for the Ramsey-Tur\'an densities $\varrho_p(pt+2)$ for $p=3,4$. Our proof translates proving these upper bounds to an extremal problem for certain weighted graphs, which is interesting in its own right. We believe our method for this weighted graph problem may be useful in systematically proving further upper bounds. So far, many of the existing upper bound proofs have followed a similar approach but in a rather ad hoc way.

\begin{theorem}\label{thm-upper-main}
	Let $t\in\mathbb{N}$. Then
	$$\varrho_3(3t+2)=\frac{5t-4}{5t+1} \quad \mbox{ and }\quad \varrho_4(4t+2)=\frac{7t-6}{7t+1}.$$
\end{theorem}

Our last upper bound shows that the bound~\eqref{eq:multi-lower} in Theorem~\ref{thm-2tri-exact} is optimal for infinitely many cases.

\begin{theorem}\label{thm-upper-easy}
	Let $p,s,t\in\mathbb{N}$ with $t(t-2) \leq s \leq t^2$ and $s+t-1 \leq p$.
	Then 
	$$\varrho_p(p+s+t-1) \leq \frac{s}{p}\left(1-\frac{1}{t}\right).$$
\end{theorem}


\subsection{Related work}

Aside from the determination of $\varrho_p(q)$, many other directions and extensions in Ramsey-Tur\'an theory have been studied. The 2001 survey of Simonovits and S\'os~\cite{SimSos} is an excellent resource for background on the area; here we confine ourselves to a brief discussion focusing on more recent developments.

This paper concerns the Ramsey-Tur\'an number $\rt_p(n,K_q,m)$ for $m=\eps n$ and $\eps \to 0$. 
For the case $p=2$ in particular, there has been a great deal of interest in other functions $m(n)$ of $n$.
Let us write
$$
{\rm ex}_q(n,m) := \rt_2(n,K_q,m)\quad\text{and}\quad {\rm ex}_q(\eps) := \lim_{n\to\infty}\frac{{\rm ex}_q(n,\eps n)}{\binom{n}{2}},
$$
and recall that the value of
$\varrho_2(q) = \lim_{\eps\to 0}{\rm ex}_q(\eps)$ is known~\cite{EHSSz,ES}.
Fox, Loh and Zhao~\cite{FLZh} showed that ${\rm ex}_4(\eps) = \varrho_2(4)+\Theta(\eps)$.
L\"uders and Reiher~\cite{LuedersReiher} extended this to obtain a formula for all $q$: they showed that ${\rm ex}_q(\eps) = \varrho_2(q)+\eps$ for odd $q$ and ${\rm ex}_q(\eps) = \varrho_2(q)+\eps-\eps^2$ for even $q$, whenever $\eps(q)$ is sufficiently small. For larger $\eps$, the situation is complicated even for the first non-trivial case $q=3$.
Mantel's theorem and an early result of Andr\'asfai~\cite{Andrasfai} determine ${\rm ex}_3(\eps)$ for $\eps \geq \frac{2}{5}$; and the regime $\eps\in(0,\frac{1}{3}]$ follows from work of Brandt~\cite{Brandt}. In a series of papers, {\L}uczak, Polcyn and Reiher determined ${\rm ex}_3(\eps)$ for various ranges of $\eps$ (see~\cite{LPR1},~\cite{LPR2}, and~\cite{LPR3} for more details). Finally, in~\cite{LPR3}, they announced a complete solution.

Determining ${\rm ex}_q(n,m)$ for functions $m(n)$ growing slower than linear has also attracted a lot of attention,
in particular determining the \emph{phase transitions} where decreasing $m(n)$ causes a large decrease in ${\rm ex}_q(n,m)$ (for a precise definition see~\cite{BHS}).
For example, ${\rm ex}_5(n,n) = \lfloor\frac{3}{4}\binom{n}{2}\rfloor$ by Tur\'an's theorem, while ${\rm ex}_5(n,o(n))=\frac{1}{2}\binom{n}{2}+o(n^2)$, so we may say there is `a phase transition at $n$'. Answering a question of Erd\H{o}s and S\'os,
Balogh, Hu and Simonovits~\cite{BHS} showed that ${\rm ex}_5(n,o(\sqrt{n\log n})) = o(n^2)$, while ${\rm ex}_5(n,c\sqrt{n\log n}) \geq \frac{1}{2}\binom{n}{2}+o(n^2)$ for infinitely many $n$ and any $c > 1$, so there is another phase transition at $\sqrt{n\log n}$.
Sudakov~\cite{Sudakov} showed that ${\rm ex}_{4}(n,e^{-\omega(n)\sqrt{\log n}}n) = o(n^2)$, while Fox, Loh and Zhao~\cite{FLZh} showed that ${\rm ex}_4(n,e^{-o(\sqrt{\log n/ \log \log n})}n) = \frac{1}{4}\binom{n}{2}+o(n^2)$. So for $q=4$, there is a phase transition somewhere between these functions.
See also~\cite{BD,KKL} for other results of this type.

Ramsey-Tur\'an type problems have also been studied for graphs other than cliques~\cite{BMSh,EHSSz,NP,Sudakov}, in hypergraphs~\cite{BL-1,ES2,FranklRodl,MubayiRodl,MubayiSos,Sidorenko}, in the multicolour setting~\cite{EHSSSz2,KKL,LNS,Schelp} and in a `counting' setting~\cite{BLSh}.
A particular tantalising open problem concerns the octahedron graph.

\begin{probx}[\cite{EHSSSz2,EHSSz,SimSos,Sudakov}]
Is $\rt_2(n,K_{2,2,2},o(n))=o(n^2)$?
\end{probx}

\medskip

\noindent
\textbf{Organisation.} The constructions for Theorems~\ref{GBE} and~\ref{thm-2tri-exact} will be given in Sections~\ref{sec-GBE} and~\ref{sec-LB-Thm1} respectively. The proofs for Theorems~\ref{thm-upper-main} and~\ref{thm-upper-easy} are in Section~\ref{sec-UB}. In Section~\ref{sec-remarks}, we give some concluding remarks.

\medskip

\noindent
\textbf{Notation.} We write $[a,b]:=\{a,\ldots,b\}\subseteq \mathbb{Z}$, for all $a,b\in \mathbb{Z}$ with $a\le b$. Whenever $a=1$, then we use $[b]$ instead of $[a,b]$. 

We will use bold face lower case symbols, e.g.~$\V{w}, \V{x}, \V{y}, \V{z}$, for vectors in $\sC^{k}$, equipped with the standard inner product $\innerp{\V{w},\V{z}}=\sum_{i\in [k]}w_i\bar{z_i}$. We write $|\V{z}|=\sqrt{\innerp{\V{z},\V{z}}}$ for its $\ell_2$-norm.

\section{Properties of high dimensional spheres}\label{sec-propertiesHDS}
In this section, we list some useful properties of high dimensional spheres, which will be used for our constructions throughout Sections~\ref{sec-GBE} and~\ref{sec-LB-Thm1}.

For $k\in\sN$, let $\sS^{k-1}(\sR)\subseteq \sR^{k}$ denote the standard $(k-1)$-dimensional real unit sphere, and write
$$\sS^{k-1}(\sC)=\left\{(z_1,\ldots, z_{k})\in\sC^{k}: \sum_{i=1}^{k}|z_i|^2=1\right\}$$
for the $(k-1)$-dimensional complex unit sphere.  As the map
\begin{equation}\label{eq: isom}
	\varphi:~ (x_1+iy_1,\ldots, x_k+iy_k)\longmapsto (x_1,y_1,x_2,y_2\ldots,x_k,y_k)
\end{equation}
from $\sS^{k-1}(\sC)$ to $\sS^{2k-1}(\sR)$ is an invertible isometry, various properties of high dimensional real spheres extend naturally to the complex ones.

Throughout the paper, when given a high dimensional unit sphere, we will write $\lambda$ for the Lebesgue measure, normalised so that the unit sphere has measure 1. For two subsets of a unit sphere $A$ and $B$, denote by $d_{\max}(A,B):=\sup\{|\V{a}-\V{b}|: \V{a}\in A,~\V{b}\in B \}$ the Euclidean distance between them. In the case $A=B$, write $\diam{A}:=d_{\max}(A,A)$ for the \emph{diameter} of $A$.  

A \emph{spherical cap} is the smaller intersection of the unit sphere with a half-space. Given a spherical cap $C$ bounded by some hyperplane $H$, we call the point in $C$ with maximum Euclidean distance to $H$ the \emph{centre} of the spherical cap. The distance from the centre to $H$ is the \emph{height} of the spherical cap. Note that $\diam{C}$ is just the diameter of the intersection of $C$ and $H$. 

We will use the following lower and upper bounds on the measure of spherical caps. They follow from the known results of the real sphere and the use of the isometry $\varphi$ in~\eqref{eq: isom}.

\begin{lemma}[\cite{BL-1}]\label{lem:cap-LB}
	For all $\delta>0$ and integers $k\ge 3$ , let $B\subseteq \sS^{k-1}(\sC)$ be the spherical cap consisting of all points with distance at most $\sqrt{2}-\delta/\sqrt{2k}$ from a fixed point in $\sS^{k-1}(\sC)$. Then $\lambda(B)\ge 1/2-\sqrt{2}\delta$.
\end{lemma}


\begin{lemma}[\cite{Tkocz}]\label{lem:cap-UB}
	Let $\alpha\in [0,1)$ and $C\subseteq \sS^{k-1}(\sC)$ be a spherical cap with height $1-\alpha$. Then $\lambda(C)\le e^{-k\alpha^2}$.
\end{lemma}

Recall that a spherical cap with height $1-\alpha$ has diameter $2\sqrt{1-\alpha^2}$.
It is a simple consequence~(see~\cite{BL-1}) of the isoperimetric inequality for spheres~\cite{Schmidt}, that for any sets $A,B,C\subseteq \sS^{k-1}(\sC)$ of equal measure, if $C$ is a spherical cap, then $d_{\max}(A,B)\ge \diam{C}$. Altogether, we have the following 2-set version of Lemma~\ref{lem:cap-UB}.

\begin{lemma}\label{lem:cap-2set}
	Let $\nu\in (0,1)$ and $A,B\subseteq\sS^{k-1}(\sC)$ with $\lambda(A), \lambda(B)>e^{-k\nu/2}$, then $d_{\max}(A,B)\ge 2-\nu$.
\end{lemma}

The following folklore result partitions the sphere into small pieces of equal measure (see e.g.~\cite{FS}).
\begin{lemma}\label{lem: partition-sphere}
	There exists $C>0$ such that the following holds. Let $0<\delta<1$ and $n\ge (C/\delta)^k$. Then $\sS^{k-1}(\sR)$ can be partitioned into $n$ pieces of equal measure, each of diameter at most $\delta$.
\end{lemma}

We also need the following geometric result, which lies at the heart of the original construction of the Bollob\'as-Erd\H os graph~\cite{BolErd76}.
\begin{theorem}[Bollob\'as-Erd\H os Rhombus lemma]\label{thm-k4}
	For all $k \in \mathbb{N}$ and all $0<\mu<1/4$, there do not exist four points $p_1,p_2,q_1,q_2\in\sS^k(\sR)$ such that $d(p_1,p_2)\ge 2-\mu$, $d(q_1,q_2)\ge 2-\mu$, and $d(p_i,q_j)\le\sqrt{2}-\mu$ for all $i,j\in [2]$.
\end{theorem}

\section{Complex Bollob\'as-Erd\H os graph}\label{sec-GBE}
Fix integers $1\le \ell <p$. For Theorem~\ref{GBE}, we will construct a graph $G$ with vertex partition $W \cup Z$ where $|W|=|Z|=n$, satisfying the following:
\stepcounter{propcounter}
\begin{enumerate}[label = {\bfseries \Alph{propcounter}\arabic{enumi}}]
	\item\label{CBE:p-ind} $\al_{p}(G)= o(n)$;
	\item\label{CBE:sparse-inside} $e(G[W]),e(G[Z]) = o(n^2)$;
    \item\label{CBE:density} $e(G)= \left(\frac{\ell}{p}-o(1)\right)n^2$;
    \item\label{CBE:clique} if $\ell\le p/2$, then $G$ is $K_{p+\ell+1}$-free.
\end{enumerate}
Corollary~\ref{cor} then follows readily by joining completely a suitable number of graphs of appropriate sizes with sublinear $p$-independence number to the above graph $G$.

\subsection{Construction}
Choose constants
\begin{align}\label{eq-hierarchy}
	0 < 1/k\ll \eps\ll 1/K\ll  1/p,\quad \mu:=\eps/\sqrt{2k}, \quad\text{ and }\quad n \geq \left(\frac{4C_{\ref{lem: partition-sphere}}}{\mu}\right)^{2k},
\end{align}
where $C_{\ref{lem: partition-sphere}}$ is the constant obtained from Lemma~\ref{lem: partition-sphere}.
Using the isometry $\varphi$ in~\eqref{eq: isom} and Lemma~\ref{lem: partition-sphere}, we can partition $\sS^{k-1}(\sC)$ into $n$ domains $D_1,\ldots,D_n$ with equal measure and  diameter at most $\frac{\mu}{4}$. Next, for all $i\in [n]$, choose two arbitrary points $\V{w}_i,\V{z}_i\in D_i$, and let $W:=\{\V{w}_1,\ldots,\V{w}_n\}$ and $Z:=\{\V{z}_1,\ldots,\V{z}_n\}$. Set $\rho:=\cos(2\pi/ p)+i\sin(2\pi/ p)$ to be the primitive $ p$-th root of unity. 
Note that $\V{w} \mapsto \rho \V{w}$ rotates $\sS^{k-1}(\mathbb{C})$.
The edge set of $G$ is defined as follows (see also Figure~\ref{fig:GBE}).
\stepcounter{propcounter}
\begin{enumerate}[label = {\bfseries \Alph{propcounter}\arabic{enumi}}]
	\item\label{it-insideW} Two vertices $\V{w}, \V{w}'\in W$ form an edge if and only if there exists $h\in [p-1]$ such that 
	$$\bigabs{\V{w}-\rho^h\V{w}'}\le \sqrt{\mu}.$$
	Define $E(G[Z])$ similarly. For such a pair, we say $\V{w}$ is an \emph{$h$-rotation} of $\V{w}'$.
	
	\item\label{it-between} A cross pair $(\V{w}_i, \V{z}_j)\in W\times Z$ forms an edge if and only if the following hold.
	\begin{itemize}
		\item[(i)] For all $h\in \sZ_p$,
		$$\bigabs{\im{\rho^h\innerp{\V{w}_i,\V{z}_j}}}\ge K\mu.$$
		
		\item[(ii)] There exists $\alpha\in [0,\frac{2\pi \ell}{p}]$ such that 
		$$e^{-i\alpha}\innerp{\V{w}_i,\V{z}_j}\in [0,1].$$
	\end{itemize}
\end{enumerate}

\begin{figure}[h]
	\centering
\scalebox{0.88}{	
	\begin{tikzpicture}	
	\filldraw[draw=none,color=gray] (0,0) -- (3,0) arc (0:120:3) -- (0,0);
		
	\begin{scope}[rotate=30]
		\draw[line width = 5mm, color=white] (0,-3) -- (0,3);
	\end{scope}
	\begin{scope}[rotate=90]
		\draw[line width = 5mm, color=white] (0,-3) -- (0,3);
	\end{scope}
	\begin{scope}[rotate=150]
		\draw[line width = 5mm, color=white] (0,-3) -- (0,3);
	\end{scope}
	
	\begin{scope}[rotate=30]
		\draw[line width = 5mm, color=Red,opacity=0.25] (0,-3) -- (0,3);
	\end{scope}
	\begin{scope}[rotate=90]
		\draw[line width = 5mm, color=Red,opacity=0.25] (0,-3) -- (0,3);
	\end{scope}
	\begin{scope}[rotate=150]
		\draw[line width = 5mm, color=Red,opacity=0.25] (0,-3) -- (0,3);
	\end{scope}
	
	\draw[thick,densely dotted,<->] (-3.5,0) -- (3.5,0);
	\draw[thick,densely dotted,<->] (0,-3.5) -- (0,3.5);
	
	\draw[<->] (0,-3.7) -- (3,-3.7);
	\draw (1.5,-3.6)  node[draw=none,fill=none,label=below:{$1$}]  {};
	
	\begin{scope}[rotate=-60]
		\draw[] (3.2,-0.25) -- (3.2,0.25);
		\draw (3.1,-0.25) -- (3.3,-0.25);
		\draw (3.1,0.25) -- (3.3,0.25);
	\end{scope}
	
	\foreach \x in {0,60,120,180,240,300} 
	{
		\draw[black] (0,0) -- (\x:3);
	}
	
	\draw[black,thick] (0,0) circle (3cm); 
	
	\draw (3.4,0)  node[draw=none,fill=none,label=right:{${\rm Re}(z)$}]  {};
	\draw (0,3.4)  node[draw=none,fill=none,label=above:{${\rm Im}(z)$}]  {};
	\draw (1.4,-3)  node[draw=none,fill=none,label=right:{$2K\mu$}]  (0) {};
	
	\draw (0.7,0) arc (0:120:0.7);
	\draw (120:0.5)  node[draw=none,fill=none,label=left:{$\frac{2\pi}{3}$}]  (0) {};
	
	\draw (0:2.8)  node[draw=none,fill=none,label=20:{$1$}]   {};
	\draw (120:2.8)  node[draw=none,fill=none,label=120:{$\rho$}]   {};
	\draw (240:2.8)  node[draw=none,fill=none,label=240:{$\rho^2$}]   {};
	
	\draw (0:3)  node[circle,inner sep=2,fill = black]   {};
	\draw (120:3)  node[circle,inner sep=2,fill = black]   {};
	\draw (240:3)  node[circle,inner sep=2,fill = black]   {};
	
\end{tikzpicture}
}
		\caption{An illustration of the position of $\innerp{\V{w}_i,\V{z}_j}$ for $p=3$. The pink stripes are the ones excluded in~\ref{it-between}(i); while the dark regions correspond to ~\ref{it-between}(ii).\label{fig:GBE}}
\end{figure}
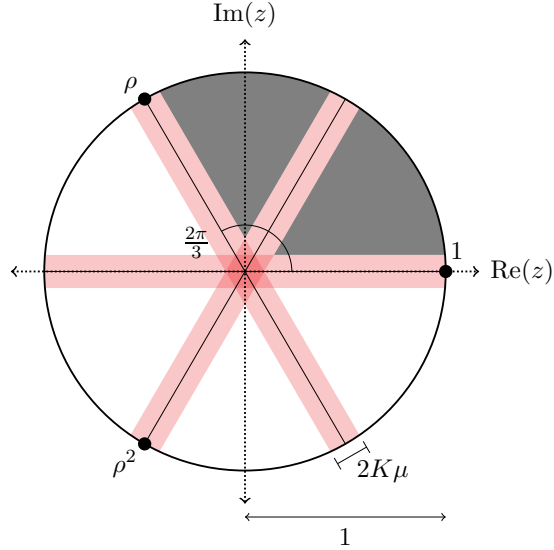

\subsection{Structure of the inner graphs}
In this subsection, using isoperimetry and concentration of measure, we shall derive that the inner graphs $G[W], G[Z]$ are $K_{p+1}$-free graphs (Lemma~\ref{obs-samerotation}) with sublinear $p$-independence number (Lemma~\ref{lem-independenceGBEoneside}) and zero edge density (Lemma~\ref{lem-maxinsidedeg}), thus verifying~\ref{CBE:p-ind} and~\ref{CBE:sparse-inside}.

\begin{lemma}\label{obs-samerotation}
	Let $\V{w}_i,\V{w}_{j},\V{w}_{t}\in W$ span a triangle in $G$. If $\V{w}_i, \V{w}_{j}$ are an $h_i$- and $h_j$-rotation of $\V{w}_{t}$ respectively, then $\V{w}_i$ is an $(h_i-h_j)$-rotation of $\V{w}_{j}$ and $h_i\neq h_j$. 
	
	Consequently, $G[W]$ is $K_{ p+1}$-free. The same holds for $Z$.
\end{lemma}

\begin{proof}
	For any $m \in [p-1]$, we have
	\begin{equation}\label{eq-largenorm}
		|1-\rho^m|^2 = 2-2\cos \left(2\pi m/p\right) \geq 2-2\cos\left(2\pi/p\right) = 4\sin^2 \left(\pi/p\right) \geq \left(4/p\right)^2,
	\end{equation}
	as $\sin x\ge \frac{2x}{\pi}$ for all $x\in[0,\frac{\pi}{2}]$ by concavity.
	By~\ref{it-insideW}, there is $h\in [p-1]$ such that $\V{w}_i$ is an $h$-rotation of $\V{w}_j$. Recall that vertices of $G$, viewed as points in $\sS^{k-1}(\sC)$, all have modulus 1.
	So
	\begin{align*}
		|1-\rho^{h_i-h_j-h}| &= |\rho^h\V{w}_j-\rho^{h_i-h_j}\V{w}_j| \leq |\V{w}_i-\rho^h\V{w}_j| + |\V{w}_i-\rho^{h_i}\V{w}_t| + |\rho^{h_i}\V{w}_t-\rho^{h_i-h_j}\V{w}_j|\\
		&= |\V{w}_i-\rho^h\V{w}_j| + |\V{w}_i-\rho^{h_i}\V{w}_t| + |\V{w}_j-\rho^{h_j}\V{w}_t| \leq 3\sqrt{\mu},
	\end{align*}
	which together with~\eqref{eq-hierarchy} and~\eqref{eq-largenorm} implies $h=h_i-h_j$. Since $h\neq 0$, we have that $h_i \neq h_j$.
	
	Suppose that $\V{w}_0,\V{w}_1,\ldots,\V{w}_{ p}$ span a clique in $G[W]$. Then by~\ref{it-insideW} there are $h_1,\ldots,h_ p \in [p-1]$ such that $\V{w}_i$ is an $h_i$-rotation of $\V{w}_0$ for all $i \in [ p]$.
	By the Pigeonhole Principle, there are distinct $i,j \in [ p]$ such that $h_i=h_j$, contradicting the first part.
\end{proof}

\begin{lemma}\label{lem-independenceGBEoneside}
	Every set $X\subseteq W$ with $|X|\ge pe^{-\mu k/40}\cdot|W|$ contains a copy of $K_p$. In particular, $$\alpha_p(G)\le 2pe^{-\mu k/40}n.$$
\end{lemma}

For its proof, we need the following consequence of concentration of measure.

\begin{lemma}\label{lem-closebyrotations}
	Let $p\ge 2$ be an integer, $\nu\le \min\{\frac{16}{p^2},1\}$ and $A\subseteq \sS^{k-1}(\sC)$ with $\lambda(A)\ge pe^{-\nu k/32}$. Then there are distinct points $\V{a}_0,\ldots,\V{a}_{p-1}\in A$ such that, for all $h,m\in \sZ_p$,
	$$\bigabs{\rho^h\V{a}_h-\rho^m\V{a}_m}<\sqrt{\nu}.$$
\end{lemma}
\begin{proof}
	For all $h\in [p-1]$, define 
	$$
	A_h:=\left\{\V{a}\in A:~\forall \V{a}'\in A,~\bigabs{\V{a}+\rho^h\V{a}'}< 2-\nu/16 \right\}.
	$$
	Note that $\lambda(A_h)\le e^{-\nu k/32}$ as otherwise, the sets $A_h$ and $-\rho^hA$ violate Lemma~\ref{lem:cap-2set}. Therefore, $A\neq A_1\cup\cdots\cup A_{p-1}$, and we can pick a point $\V{a}_0\in A\setminus(A_1\cup\cdots\cup A_{p-1})$. 
	
	For all $h\in [p-1]$, since $\V{a}_0\notin A_h$, there exists a point $\V{a}_h\in A$
	with  $|\V{a}_0+\rho^h\V{a}_h|\ge 2-\nu/16$. We claim that $|\V{a}_0-\rho^h\V{a}_h|<\sqrt{\nu}/2$, for all $h\in \sZ_p$. The inequality is trivial for $h=0$. For $h\in [p-1]$, it follows from the parallelogram law:
	\begin{equation*}
		|\V{a}_0-\rho^h\V{a}_h|^2=  2(|\V{a}_0|^2+|\rho^h\V{a}_h|^2)-|\V{a}_0+\rho^h\V{a}_h|^2\le 4-(2-\nu/16)^2<\nu/4.
	\end{equation*}
	Thus, for all $h,m\in \sZ_p$, by the triangle inequality, we obtain
	\begin{equation}\label{eq-smallrotationdistance}
		|\rho^h\V{a}_h-\rho^m\V{a}_m|\le|\rho^h\V{a}_h-\V{a}_0|+|\V{a}_0-\rho^m\V{a}_m|<\sqrt{\nu}.
	\end{equation}

	We are left to show that all points $\V{a}_0,\ldots, \V{a}_{p-1}$ are distinct. Suppose to the contrary that for some distinct $h,m\in \sZ_p$, $\V{a}_h=\V{a}_m$. Then, as in~\eqref{eq-largenorm},
	\begin{equation*}
		|\rho^h\V{a}_h-\rho^m\V{a}_m|=|\rho^h-\rho^m|=|\rho^{h-m}-1|\ge  4/p\ge \sqrt{\nu},
	\end{equation*}   
	a contradiction to~\eqref{eq-smallrotationdistance}.
\end{proof}

We are now ready to prove Lemma~\ref{lem-independenceGBEoneside}.
\begin{proof}[Proof of Lemma~\ref{lem-independenceGBEoneside}]
Let $X\subseteq W$ with $|X|\ge pe^{-\mu k/40}\cdot|W|$
	and let $A:=\bigcup\{D_i:\V{w}_i\in X\}$; then $\lambda(A)=|X|/|W|\ge pe^{-\mu k/40}$. By Lemma~\ref{lem-closebyrotations}, there exist distinct $\V{a}_0,\ldots,\V{a}_{p-1}\in A$ such that $|\rho^h\V{a}_h-\rho^m\V{a}_m|\le \sqrt{4\mu/5}$ for all $h,m\in \sZ_p$. For each $h\in \sZ_p$, let $\V{x}_h\in X$ be the vertex lying in the same domain $D_i$ as $\V{a}_h$, i.e. $|\V{x}_h-\V{a}_h|\le \mu/4$. 
	
	Now, for distinct integers $h,m\in \sZ_p$, by triangle inequality and~\eqref{eq-hierarchy}, we get
	\begin{equation*}
		|\V{x}_h-\rho^{m-h}\V{x}_m|=|\rho^h\V{x}_h-\rho^m\V{x}_m|\le |\rho^h\V{x}_h-\rho^h\V{a}_h|+|\rho^h\V{a}_h-\rho^m\V{a}_m|+|\rho^m\V{a}_m-\rho^m\V{x}_m|\le\sqrt{\mu}.
	\end{equation*}
	Thus $\V{x}_h\V{x}_m\in E(G[W])$ and $\V{x}_0,\ldots,\V{x}_{p-1}$ induce a copy of $K_p$, finishing the proof.
\end{proof}

The fact that the inner graphs $G[W],G[Z]$ have zero edge density follows already from their being $K_{p+1}$-free and having sublinear $p$-independence number, as then their maximum degree is at most $\alpha_p(G[W])=e^{-\Theta(\sqrt{k})}n$ by Lemma~\ref{lem-independenceGBEoneside}. We can in fact give a tighter bound via a direct estimation. 
\begin{lemma}\label{lem-maxinsidedeg}
	The maximum degree of $G[W], G[Z]$ is at most $pe^{-k(1-\mu)^2}n\le e^{-k/2}n$. 
\end{lemma}

\begin{proof}
	By the construction of $G$, in particular~\ref{it-insideW} and that each domain $D_i$ has diameter at most $\mu/4$, we see that in $G[W]$ every vertex has degree $n(p-1)$ times the measure of a spherical cap whose points are within distance $d=\sqrt{\mu}\pm \mu/4$ from its centre. As the height of such a cap is precisely $d^2/2$, the conclusion follows from Lemma~\ref{lem:cap-UB}.
\end{proof}

\subsection{Angle decides cross density}
In this subsection, we verify~\ref{CBE:density}.
\begin{lemma}\label{lem-mindegree}
	Every vertex in $W$ has $(\frac{\ell}{p}\pm \frac{1}{\sqrt{K}}) n$  neighbours in $Z$, and vice versa.
\end{lemma}

\begin{proof}
For each point $\V{x}\in \sS^{k-1}(\sC)$, define sets
\begin{align*}
	J(\V{x})&:=\left\{\V{y}\in \sS^{k-1}(\sC): \bigabs{\iminnerp{\V{x}, \rho^h \V{y}}}>\textstyle{\left(K+\frac{1}{4}\right)}\mu, \text{ for all } h\in \sZ_p\right\},\text{ and }\\
	I(\V{x})&:=\left\{\V{y}\in J(\V{x}): \arg\innerp{\V{x},\V{y}}\in \left[1/\sqrt{K},~2\pi \ell/p-1/\sqrt{K}\right]\right\}.
\end{align*} 
We shall show that for each vertex $\V w\in W$, vertices whose domains intersect the associated set $I(\V w)$ are adjacent to $\V w$. That is, 
\begin{equation}\label{eq:Zw}
	Z_\V{w}:=\{\V{z}_j\in Z: D_j\cap I(\V{w})\neq \varnothing\}\subseteq N_G(\V w).
\end{equation}

We first bound the measure of this associated set.
\begin{claim}\label{lem-crossnbrs}
	For any $\V{x}\in \sS^{k-1}(\sC)$, we have $\lambda(I(\V{x}))\ge \frac{\ell}{p}-\frac{1}{\sqrt{K}}$.
\end{claim}
\begin{poc}
	Fix an arbitrary $\V{x}\in \sS^{k-1}(\sC)$. Define 
	$$L:=\left\{\V{y}\in \sS^{k-1}(\sC): |-\V{x}i-\V{y}|\le\sqrt{2}-K\mu\right\}.$$
	By Lemma~\ref{lem:cap-LB} (with $\delta=\eps K$), we see that $\lambda(L)\ge \frac{1}{2}-\sqrt{2}\eps K$. For each $y\in L$, using~\eqref{eq-hierarchy},
	\begin{equation*}
		2- 2\,\iminnerp{\V{x},\V{y}}=2-2\,\reinnerp{-\V{x}i,\V{y}}=|- \V{x}i-\V{y}|^2\le (\sqrt{2}-K\mu)^2<2-2\left(K+\tfrac{1}{4}\right)\mu,
	\end{equation*}
    and so  $\iminnerp{\V{x},\V{y}}>(K+\frac{1}{4})\mu$. Thus, by symmetry, the measure of the set
	$$
	\left\{\V{y}\in \sS^{k-1}(\sC): \bigabs{\iminnerp{\V{x},\V{y}}}\le\left(K+\tfrac{1}{4}\right)\mu\right\}$$
	is at most $1-2\lambda(L)\le 1-2\left(\frac{1}{2}-\sqrt{2}\eps K\right)=2\sqrt{2}\eps K$.
	Since $\innerp{\V{x},\rho^h\V{y}}=\innerp{\rho^{-h}\V{x},\V{y}}$, we can take the union bound of such sets over $\V{x}$, $\rho\V{x}$,$\ldots$, $\rho^{p-1}\V{x}$ to deduce that $\lambda(\overline{J(\V{x})})\le 2\sqrt{2}\eps Kp$, where $\overline{J(\V{x})}:=\sS^{k-1}(\sC)\setminus J(\V{x})$. Therefore, by~\eqref{eq-hierarchy}, we get
	\begin{equation*}
		\lambda(I(\V{x}))\ge\frac{1}{2\pi}\left(\frac{2\pi\ell}{p}-\frac{2}{\sqrt{K}}\right)-\lambda(\overline{J(\V{x})})\ge \frac{\ell}{p}-\frac{1}{\sqrt{K}}.
	\end{equation*}
\end{poc}

	Fix a vertex $\V{w}\in W$ and let $Z_{\V w}$ be as in~\eqref{eq:Zw}. As each domain $D_j$ has measure $\frac{1}{n}$, the above claim entails
	\begin{equation*}
		|Z_\V{w}|\ge n\sum_{\V{z}_{j}\in Z_\V{w}}\lambda(D_j\cap I(\V{w}))=n\cdot\lambda(I(\V{w}))\ge \left(\frac{\ell}{p}-\frac{1}{\sqrt{K}}\right)n.
	\end{equation*}  
    We are left to show $Z_{\V w}\subseteq N_G(\V w)$ and that the upper bound on the degrees can be obtained similarly.
	
	Fix a vertex $\V{z}_j\in Z_\V{w}$ and take a point $\V{z}^*\in D_j\cap I(\V{w})$. As $\V z_j,\V z^*\in D_j$, $|\V{z}^*-\V{z}_j|\le \mu/4$. For any $h\in \sZ_p$, using the triangle inequality, that $\V z^*\in I(\V{w})$, and the Cauchy–Schwarz inequality, we see that
	\begin{align*}
		|\im{\rho^h \innerp{\V{w},\V{z}_j}}|=|\iminnerp{\V{w},\rho^{-h}\V{z}_j}|&\ge |\iminnerp{\V{w},\rho^{-h}\V{z}^*}|-|\iminnerp{\V{w},\rho^{-h}\V{z}^*}-\iminnerp{\V{w},\rho^{-h}\V{z}_j}|\\
		&\ge (K+\tfrac{1}{4})\mu-|\innerp{\V{w},\rho^{-h}\V{z}^*}-\innerp{\V{w},\rho^{-h}\V{z}_j}|\\
		&\ge(K+\tfrac{1}{4})\mu-|\V{w}|\cdot|\rho^{-h}\V{z}^*-\rho^{-h}\V{z}_{j}|\\
		&=(K+\tfrac{1}{4})\mu-|\V{z}^*-\V{z}_{j}|\ge K\mu,
	\end{align*}  
	verifying~\ref{it-between}(i). 
	
	For~\ref{it-between}(ii), let $x:=\innerp{\V{w},\V{z}_j}$ and $y:=\innerp{\V{w},\V{z}^*}$. Then as above we have $|y|\ge |\iminnerp{\V{w},\V{z}^*}|\ge (K+\frac{1}{4})\mu$, $|x-y|\le |\V w||\V z^*-\V z_j|\le \mu/4$ and so $|x|\ge |y|-|x-y|\ge K\mu$.
	Consider the triangle in the complex plane with vertices $0,x,y$ and let $\alpha$ be its angle at $0$. Then the law of cosines implies that
	$$\cos\alpha=\frac{|x|^2+|y|^2-|x-y|^2}{2|x||y|}\ge\frac{(1-\frac{1}{8K})(|x|^2+|y|^2)}{2|x||y|}\ge 1-\frac{1}{8K}.$$
	As $\cos \theta\le 1-\frac{\theta^2}{4}$ when $\theta\in(0,1)$, we get $|\arg\innerp{\V{w},\V{z}_j} - \arg\innerp{\V{w},\V{z}^*}|=\alpha\le 1/\sqrt{2K}$. Therefore $\arg\innerp{\V{w},\V{z}_j}\in [0,2\pi\ell/p]$ as $\V z^*\in I(\V w)$, implying~\ref{it-between}(ii).
	
	We omit the proof of the upper bound on degrees since it is very similar but easier, noting that the upper bound corresponding to Claim~\ref{lem-crossnbrs} is clear as the set of $\V{y} \in \sS^{k-1}(\mathbb{C})$ with $\arg\innerp{\V{x},\V{y}} \in [s,t]$ trivially has measure at most $\frac{t-s}{2\pi}$.
\end{proof}

\subsection{Clique number of $G$}
Finally, we verify~\ref{CBE:clique}. 
We will need the following lemma. 

\begin{lemma}\label{lem-farpairs}
	Let $U$ be a subset of $W$ or $Z$ such that $G[U]$ is a clique. Then there exists a pair of vertices $\V{u},\V{u}'\in U$ such that $\V{u}$ is an $h$-rotation of $\V{u}'$ with $\min\{|U|-1, \floor{p/2}\}\le h\le \floor{p/2}$.
\end{lemma}
\begin{proof}
	Note that the upper bound is trivial, as in every pair of adjacent vertices, taking the smaller angle we see that one is an $h$-rotation of the other for some $h\le \floor{p/2}$. For the lower bound, set $u:=|U|$. We will prove the case when $u-1\le \floor{p/2}$. The other case can be reduced to this case by taking a subset of $U$ of size $\floor{p/2}+1$. 
	
	Fix a vertex $\V{u}_0\in U$. By~\ref{it-insideW},  every vertex $\V{u}\in U\setminus\{\V{u}_0\}$ is an $h$-rotation of $\V{u}_0$, for some $h\in[p-1]$. We may assume that $h\in \{-u+2,\ldots,u-2\}\setminus\{0\}$, for otherwise, $\V{u}_0,\V{u}$ is the pair we seek. Pair up all elements in $\{-u+2,\ldots,u-2\}\setminus\{0\}$ such that their difference is $u-1$, that is, partition it into pairs $\{-u+1+j,j\}$, for all $j\in [u-2]$. We say a vertex in $U\setminus\{\V{u}_0\}$ \emph{belongs to the $j$-th pair} if it is either a $(-u+1+j)$-rotation or a $j$-rotation of $\V{u}_0$. By the Pigeonhole Principle, there exists a pair of vertices $\V{u},\V{u}'\in U\setminus\{\V{u}_0\}$ that form the $j$-th pair for some $j\in[u-2]$. By Lemma~\ref{obs-samerotation}, they are not the same rotation of $\V{u}_0$. Therefore, without loss of generality, we can assume that $\V{u}$ is a $j$-rotation of $\V{u}_0$ and $\V{u}'$ is a $(-u+1+j)$-rotation of $\V{u}_0$. Thus, by Lemma~\ref{obs-samerotation}, $\V{u}$ is a $(u-1)$-rotation of $\V{u}'$.
\end{proof}

Take $\ell\le p/2$. Suppose to the contrary that $G$ contains a copy of $K_{p+\ell+1}$ on vertex set $X\cup Y$, with, say, $X\subseteq W$, $Y\subseteq Z$ and $|X|\ge |Y|$.
So $\frac{p}{2} < |X| \leq p$ (since $X$ is $K_{p+1}$-free) and hence $|Y| \geq \ell+1$.
Let $\V{x}_0 \in X$ be arbitrary, set
$$
A := \{h \in \mathbb{Z}_p: X \text{ contains an } h\text{-rotation of }\V{x}_0\}\quad\text{and}\quad B := \{\rho^h: h \in A\} \subseteq \{z \in \mathbb{C} : |z|=1\}.
$$
If the closed convex hull of $B$ fails to contain $0$, then $B$ is contained in an open half-plane of $\mathbb{C}$ bounded by a line passing through $0$, so we have $|X|=|B| \leq \frac{p}{2}$, a contradiction.
Thus $0$ lies in the closed convex hull of $B$. Consequently, by Carath\'eodory's theorem, there are $h_1,h_2,h_3 \in A$ such that $0$ is in the closed triangle with vertices $\rho^{h_1},\rho^{h_2},\rho^{h_3}$.
In other words, there are reals $\lambda_1,\lambda_2,\lambda_3 \in [0,1]$ such that $\lambda_1+\lambda_2+\lambda_3 =1$ and $\lambda_1\rho^{h_1}+\lambda_2\rho^{h_2}+\lambda_3\rho^{h_3}=0$.
Pick $\V{x}_1,\V{x}_2,\V{x}_3 \in X$ such that $|\V{x}_i-\rho^{h_i}\V{x}_0| \leq \sqrt{\mu}$ for all $i \in [3]$.
By the triangle inequality
$$
\V{x} := \lambda_1\V{x}_1+\lambda_2\V{x}_2+\lambda_3\V{x}_3 = \lambda_1(\V{x}_1-\rho^{h_1}\V{x}_0)+\lambda_2(\V{x}_2-\rho^{h_2}\V{x}_0)+\lambda_3(\V{x}_3-\rho^{h_3}\V{x}_0) 
$$
satisfies $|\V{x}| \leq \sqrt{\mu}$.
By Lemma~\ref{lem-farpairs}, there exist a pair of vertices $\V{y},\V{y}' \in Y$ and an integer $m$ such that $\ell \leq m \leq \frac{p}{2}$ and $|\V{y}-\rho^m\V{y}'| \leq \sqrt{\mu}$. Now,
\begin{equation}\label{eq-xyy'}
\iminnerp{\V{x},\V{y}-\rho^m\V{y}'} \leq |\innerp{\V{x},\V{y}-\rho^m\V{y}'}| \leq |\V{x}||\V{y}-\rho^m\V{y}'| \leq \mu.
\end{equation}
 On the other hand,~\ref{it-between}(ii) implies that $0 \leq \arg\innerp{\V{x}_i,\V{y}}, \arg\innerp{\V{x}_i,\V{y}'} \leq \frac{2\pi\ell}{p}$. So
 $$
 0 \leq \arg(-\rho^{-m}\innerp{\V{x}_i,\V{y}'}) \leq 2\pi\left(\frac{\ell}{p}+\frac{p/2-m}{p}\right) \leq \pi
 $$
 since $\ell \leq m$.
 So the imaginary parts of $\innerp{\V{x}_i,\V{y}},-\rho^{-m}\innerp{\V{x}_i,\V{y}'}$ are positive, and
 \ref{it-between}(i) yields
$$
\iminnerp{\V{x}_i,\V{y}}, \im{-\rho^{-m}\innerp{\V{x}_i,\V{y}'}} \geq K\mu\quad\text{for all }i \in [3], 
$$
whence $\iminnerp{\V{x},\V{y}-\rho^m\V{y}'} \geq 2K\mu$.
This contradiction to~(\ref{eq-xyy'}) concludes the proof of~\ref{CBE:clique}.


\section{Multipartite Bollob\'as-Erd\H{o}s graph}\label{sec-LB-Thm1}
In this section, we prove Theorem~\ref{thm-2tri-exact}. Let $\ell,p,q$ be positive integers where $q$ is even and $\ell \geq p(q-1)$. We will construct an $n$-vertex graph $G$ satisfying the following:
\stepcounter{propcounter}
\begin{enumerate}[label = {\bfseries \Alph{propcounter}\arabic{enumi}}]
	\item\label{it:MBE-ind} $\al_{2^\ell}(G)= o(n)$;
	\item\label{it:MBE-inner} $G$ can be made $q$-partite by removing $o(n^2)$ edges;
	\item\label{it:MBE-edges} $e(G)=\left(\frac{2^p(q-1)}{2^{\ell+1} q}-o(1)\right)n^2$;
	\item\label{it:MBE-clique} $G$ is $K_{2^\ell+2^p+q-1}$-free.
\end{enumerate}
This will show that $\varrho_{2^\ell}(2^\ell+2^p+q-1) \geq \frac{2^p(q-1)}{2^\ell q}$. 


\subsection{Construction}
Choose constants
\begin{align}
0 < 1/m \ll 1/k \ll \eps \ll 1/\ell,1/p,1/q \quad\text{and}\quad\mu:=\eps/\sqrt{k}.
\end{align}
Partition $\sS^{k}(\sR)$ into $m$ domains $D_1,\ldots,D_m$ with equal measure and  diameter at most $\frac{\mu}{4}$. Next, let $P\subseteq \sS^{k}(\sR)$ be an arbitrary set of $m$ points with exactly one point from each domain. 

We will construct a multipartite analogue $G$ of the Bollob\'as-Erd\H{o}s graph in two stages. Our construction is inspired by ideas from~\cite{BL-1,BL-2, Rodl}, which themselves build on the Bollob\'as-Erd\H os graph. The graph $G$ will be built on vertex set $V_1\cup\ldots\cup V_q$. For the inner edges, each $G[V_i]$, $i\in [q]$, will be isomorphic to a high dimensional Borsuk graph $B(\ell)$, which is defined via a certain auxiliary hypergraph $\cB$ encoding geometric information about cliques in $B(\ell)$; see Section~\ref{sec-hdB}. The adjacencies between $V_i$ and $V_j$ are more involved; roughly speaking, cross edges are set up with certain geometric constraints on their endpoints (depending on $(i,j)$); see Section~\ref{sec-mBE}. These geometric constraints will be used in the rest of this section, together with the properties of the high dimensional Borsuk graph, to bound the clique number and prove various other properties of $G$.

\subsubsection{High dimensional Borsuk graph}\label{sec-hdB}
We describe the \emph{$\ell$-dimensional Borsuk graph} $B(\ell)$ with vertex set $V(B(\ell))\subseteq \prod_{h\in[\ell]}S_h$,
where each $S_h$ is a copy of $\sS^k(\sR)$. 
We will define $B(\ell)$ via a sequence of auxiliary (hyper)graphs 
$$\{Q_h\}_{h\in[\ell]}\longrightarrow \cB \longrightarrow \cB'\longrightarrow B(\ell).$$	
This part of the construction follows~\cite{BL-2}.
First, let
\begin{align}
r := 2^\ell\quad\text{and}\quad\zeta:=\exp\left(-\frac{k\mu}{3\cdot2^{2\ell}}\right),
\end{align} 
chosen so that a spherical cap of diameter at most $2-\eps/(2\sqrt{k})$ has measure at most $\zeta$.

\medskip
\noindent
\textbf{Step 1: $\{Q_h\}_{h \in [\ell]}$.}
Let $Q_0$ be a copy of $K_{r}$ with $V(Q_0):=\lbrace b^{(1)},\ldots,b^{(r)}\rbrace$, where each $b^{(i)}$ is a distinct binary string of length $\ell$. 
For each $h \in [\ell]$, let $Q_h\cong K_{r/2,r/2}$ be the spanning subgraph of $Q_0$ in which the two partite sets consist of vertices with $0$ and $1$ respectively in the $i$-th coordinate. 
Thus $Q_0=\bigcup_{h\in [\ell]}Q_h$ (note that this union is not edge-disjoint). 

\medskip
\noindent
\textbf{Step 2: $\cB$.}	
Next, we construct an $r$-uniform hypergraph $\cB$ on vertex set
$$
V(\cB) := P^\ell\subseteq  \prod_{h\in[\ell]}S_h.
$$
We write each vertex $\V{v}\in P^{\ell}$ as $\V{v} = (v_1,\ldots,v_\ell)$, where $v_h \in S_h$ is the projection of $\V{v}$ on $S_h$. Then  
\begin{align}
\{\V{v}^{(1)},\ldots,\V{v}^{(r)}\} \in E(\cB) \iff~ &\text{for all }h\in [\ell] \text{ and }i,j\in [r], \text{ whenever }b^{(i)}b^{(j)}\in E(Q_h), \nonumber\\
&\text{we have }\bigabs{v^{(i)}_h - v^{(j)}_h}\ge 2-\mu. \label{def:cB}
\end{align}
In other words, the projections of $\V{v}^{(i)}$ and $\V{v}^{(j)}$ onto $S_h$ are almost antipodal. 
Note that the definition depends on the labelling of the vertices within the hyperedge.

\medskip
\noindent
\textbf{Step 3: $\cB'$.}		
To construct the hypergraph $\cB'$, we apply a theorem of Balogh and Lenz~\cite[Theorem~16]{BL-1}.
Rather than state the theorem in its fully generality, we only state its conclusion when applied to $\cB$ (with their $(\gamma, t)$ being our $(\zeta, t^{1/\ell})$ here):

There exist $t \in \mathbb{N}$ and an $r$-uniform hypergraph $\cB'$ such that the following holds.
Given $v \in P$, let $R(v)$ be a set of $t^{1/\ell}$ distinct points from the same domain $D_i$ as $v$ which are arbitrarily close to $v$.
We say that $\V{u} = (u_1,\ldots,u_{\ell})$ \emph{corresponds} to $\V{v} \in V(\cB)$ if $u_h \in R(v_h)$ for all $h \in [\ell]$. 
Then
$$
V(\cB') := \{\V{u} : \V{u} \text{ corresponds to some } \V{v} \in V(\cB)\} \subseteq \prod_{h \in [\ell]}S_h
\quad\text{ (so }|V(\cB')| = m^\ell t),
$$
and $\cB'$ satisfies the following:	
\begin{itemize}
	\item Fix an arbitrary edge $\lbrace \V{v}^{(1)},\ldots, \V{v}^{(r)}\rbrace\in E(\cB)$. For all $i\in [r]$, let $U_i$ be a set of at least $\zeta t$ vertices of $\cB'$ corresponding to $\V{v}^{(i)}$. Then the hypergraph $\cB'$ contains at least one hyperedge with exactly one vertex in each $U_i$. 
	\item There is no subhypergraph $\cB''\subseteq \cB'$ with $|V(\cB'')|\le r^3$, $|E(\cB'')| \geq 1$ and $|V(\cB'')|+(1+\zeta-r)(|E(\cB'')|-1)<r$. (Informally, all subhypergraphs of $\cB'$ are sparse.) 
\end{itemize}

We remark that the proof of~\cite[Theorem~16]{BL-1} uses an idea of R\"odl~\cite{Rodl}. Essentially, the claimed hypergraph is obtained by first blowing up the original one, then taking a random subhypergraph of it and finally using a first moment deletion method to get rid of small configurations.


Note that $\cB'$ has the same geometric properties as $\cB$ in the sense that every hyperedge $\{\V{v}^{(1)},\ldots,\V{v}^{(r)}\} \in E(\cB')$ maintains the property on the right-hand side of~(\ref{def:cB}).


\medskip
\noindent
\textbf{Step 4: $B(\ell)$.}	
The high dimensional Borsuk graph $B(\ell)$ is defined to be the \emph{shadow graph} of $\cB'$,  that is, 
\begin{align*}
&V(B(\ell)):=V(\cB') \quad\text{and}\\
&\V{u}\V{v} \in E(B(\ell)) \iff \{\V{u},\V{v}\}\text{ is contained in some hyperedge of }\cB'.
\end{align*}
Note that the standard Borsuk graph is precisely $B(1)$.

It was proved in~\cite[Lemmas~6, 14]{BL-2} (using the second bullet point above) that
\begin{equation}\label{eq:shadow}
	\text{ every set $A$ of vertices that spans a clique in $B(\ell)$ lies in some hyperedge of $\cB'$,}
\end{equation}
implying that $B(\ell)$ is $K_{2^{\ell}+1}$-free; and furthermore, $B(\ell)$ has sublinear $2^\ell$-independence number: $\alpha_{2^\ell}(B(\ell)) \leq  r^\ell2^{\ell+r+1}\zeta\cdot  |B(\ell)|$.

%
%


\subsubsection{The final graph $G$}\label{sec-mBE}
Let $n:=m^\ell tq$ (where $t$ was defined in Step~3).
We are now ready to construct the $n$-vertex multipartite Bollob\'as-Erd\H{o}s graph $G$.
We let
$$
V(G) := V_1 \cup \ldots \cup V_q\quad\text{where}\quad G[V_i] \text{ is a copy of }B(\ell) \text{ for all }i \in [q].
$$ 

For cross edges, recall that each vertex $\V{v}=(v_1,\ldots, v_\ell)\in V(G)$ is a length-$\ell$ vector where $v_h\in S_h$, for all $h\in [\ell]$. The adjacencies of pairs of vertices in $G$ between partite classes are determined by the distances of their projections onto certain blocks of coordinates. The relevant blocks for each pair $(V_i,V_{i'})$, $\{i,i'\}\in{[q]\choose 2}$, come from an edge-colouring, as follows. 
Let $\chi$ be a proper $(q-1)$-edge colouring of the $q$-clique on $\lbrace V_1,\ldots,V_q\rbrace$, with colours $\lbrace 1,\ldots,q-1\rbrace$ (so each colour class is a perfect matching, and $\chi$ exists since $q$ is even).
For brevity, let $c_{ij} := \chi(V_iV_j)$ for all $\{i,j\}\in{[q]\choose 2}$.

For $h,h'\in[\ell]$ and $\{i,i'\}\in{[q]\choose 2}$, we say that coordinates $(h,h')$ are \emph{$(i,i')$-related} if there exist $j \in [q]\setminus\lbrace i,i'\rbrace$ and $s \in [p]$ such that 
\begin{itemize}
	\item either $(h,h')=((c_{ij}-1)p+s,(c_{i'j}-1)p+s)$,
	
	\item  or $h = h' > p(q-1)$.
\end{itemize}

Finally, given $\V{u}  \in V_i$ and $\V{v}  \in V_{i'}$, we define
\begin{align*}
\V{u}\V{v} \in E(G[V_i,V_{i'}]) \quad \iff \quad &|u_{h}-v_{h'}| \leq \sqrt{2}-\mu \text{ whenever }(h,h') \text{ are } (i,i')\text{-related}.
\end{align*}
Informally, the projection $u_h$ is in the hemisphere centred at $v_{h'}$, and vice versa.

\medskip

This completes the construction of $G$.

\medskip

By construction, $\alpha_{2^{\ell}}(G)\le q\cdot\alpha_{2^{\ell}}(B_{\ell})$. Recall that $B(\ell)$ is $K_{2^{\ell}+1}$-free and has sublinear $2^{\ell}$-independence number, thus verifying~\ref{it:MBE-ind}. Note that~\ref{it:MBE-inner} then follows immediately as $\Delta(B(\ell))\le\alpha_{2^\ell}(B(\ell))=o(|B_{\ell}|)$ due to $K_{2^{\ell}+1}$-freeness.

We will spend the rest of the section verifying~\ref{it:MBE-edges} and~\ref{it:MBE-clique}.

\subsection{Cross density}
In this subsection, we verify~\ref{it:MBE-edges}. We will use the following easy observation about how coordinates are related between different $V_i$'s.
\begin{observation}\label{obs-relatedcoord}
	Let $i,i'\in [q]$ be distinct. For all $h \in [\ell]$, there is at most one $h' \in [\ell]$ such that $(h,h')$ are $(i,i')$-related. 
	Furthermore, for all but exactly $p$ values of $h\in [\ell]$, there exists a unique $h'\in[\ell]$ such that $(h,h')$ are $(i,i')$-related.
\end{observation}

\begin{proof}
	Suppose that $(h,h')$ and $(h,h'')$ are both $(i,i')$-related, where $h' \neq h''$. Then
	there exist $j,j' \in [q]\setminus\lbrace i,i'\rbrace$ and $s \in [p]$ such that $(c_{ij}-1)p+s = h = (c_{ij'}-1)p+s$.
	Thus $c_{ij}=c_{ij'}$. Since $\chi$ is a proper edge-colouring, we have $j=j'$. Then $h' = (c_{i'j}-1)p+s = h''$, a contradiction.
	
	To prove the second part, note that, by definition, we have $\lbrace c_{ij} : j \in [q]\setminus\lbrace i,i'\rbrace \rbrace = [q-1]\setminus\lbrace c_{ii'}\rbrace$. Therefore, if some $h\in [\ell]$ is not $(i,i')$-related to any other $h'\in [\ell]$, then $h=(c_{ii'}-1)p+s$ for some $s \in [p]$.
\end{proof}

Fix arbitrary distinct $i,i'\in [q]$ and a vertex $\V{v}=(v_1,\ldots,v_\ell)\in V_i$. We will compute the degree of $\V{v}$ to $V_{i'}$. By
Observation~\ref{obs-relatedcoord}, there are exactly $\ell-p$ pairs $(h,h')\in [\ell]^2$ which are $(i,i')$-related. Fix an arbitrary such pair $(h,h')$. Define
\begin{align*}
I&:=\left\{j\in [m]:d_{\max}(D_j,v_h)\le \sqrt{2}-\mu\right\},\quad\text{and}\quad\\
L&:=\left\{y\in \sS^k(\sR):|y-v_h|\le \sqrt{2}-\frac{3}{4}\mu\right\}.
\end{align*}
Recall that $\mu = \eps/\sqrt{k}$ and $1/k \ll \eps$ (so $k\mu$ is very large while $\sqrt{k}\mu$ is very small).
By Lemma~\ref{lem:cap-LB}, we have $\lambda(L)\ge  \frac{1}{2}-2\eps$. Since all domains have equal measure and diameter at most $\frac{\mu}{4}$, we have
$|I|\ge\left(\frac{1}{2}-2\eps\right)m$.
Recall that $V(\cB')$ is a subset of $\prod_{h \in [\ell]}S_h$ where for each $\V{v} \in V(\cB)$, there are $t$ vertices of $\cB'$ where the $h$-th coordinate of each one is a distinct point of $D_h$, for all $h \in [\ell]$.
Also, $V_{i'}$ is a copy of $V(B(\ell))=V(\cB')$.
As there are exactly $\ell-p$ many $(i,i')$-related pairs, the number of vertices $\V{u}\in V_{i'}$ such that $|v_{h}-u_{h'}|\le \sqrt{2}-\mu$ is at least $\left(\frac{1}{2}-2\eps\right)m^\ell t$ and clearly at most $\frac{1}{2}m^\ell t$. Thus there are $\left(\frac{1}{2}\pm 2\eps\right)^{\ell-p}m^\ell t$ vertices $\V{u}\in V_{i'}$ that are adjacent to $\V{v}$. Hence
\begin{align*}
e(G)&= \left(\frac{1}{2}\pm 2\eps\right)^{\ell-p}\cdot\frac{m^\ell  t  (q-1)}{2}\cdot n=\left(\frac{1}{2}\pm 2\eps\right)^{\ell-p} \cdot\frac{q-1}{2q}\cdot n^2= \left(\frac{q-1}{2q}\cdot 2^{p-\ell}\pm 4(\ell-p)\eps\right)  n^2,
\end{align*}	
thus verifying~\ref{it:MBE-edges}.




\subsection{Clique number of $G$}
In this subsection, we verify~\ref{it:MBE-clique}. We will use the following simple fact about the graphs $Q_h$.
\begin{observation}\label{obs-indset}
	For all $I \subseteq [\ell]$, $\alpha\left(\bigcup_{h \in I}Q_h\right)=2^{\ell-|I|}$. 
\end{observation}

\begin{proof}
	Let $T := \bigcup_{h \in I}Q_h$.
	If $bb' \notin E(Q_h)$, then $b_h= b'_h$. So the set $\{ b : b_h = 1 \text{ for all }h \in I\}$ is an independent set in $T$ of size $2^{\ell-|I|}$.
	On the other hand, choose an arbitrary subset  $X \subseteq V(T)$ with $|X| \geq 2^{\ell-|I|}+1$.
	Then there are vertices $bb' \in X$ which differ at some coordinate $h \in I$, and so $bb' \in E(T)$.
\end{proof}


\begin{definition}
	A coordinate $h\in[\ell]$  is \emph{lengthy} for a vertex subset $A\subseteq V(B(\ell))$ if there exist two vertices $\V{v},\V{v'}\in A$ with $|v_h - v'_h|\ge 2-\mu$ (i.e.~whose projections onto $S_h$ are almost antipodal).
\end{definition}

For the rest of this subsection, fix a set $A$ of vertices which span a clique in $G$. Given $X \subseteq [\ell]$, define $L_i(X)$ to be the set of lengthy coordinates $h \in X$ for $V_i \cap A$. The following lemma helps us relate the number of lengthy coordinates for a clique to its size.

\begin{lemma}\label{lem-cliquesize}
	For all  $i\in [q]$, we have $|V_{i}\cap A|\le 2^{|L_i([\ell])|}$.
\end{lemma}
\begin{proof}
	Write $s := |V_i \cap A|$.
	Since $G[V_i\cap A]\subseteq B(\ell)$ is a clique, by~\eqref{eq:shadow}, the vertex subset $V_i\cap A$ lies in some hyperedge of $\cB'$. Recall that hyperedges of $\cB'$ satisfy the right-hand side of~(\ref{def:cB}). Write $V_i\cap A =: \{\V{v}^{(1)},\ldots,\V{v}^{(s)}\}$. Then for all $h \in [\ell]$ and $i,j \in [s]$, whenever $b^{(i)}b^{(j)} \in E(Q_h)$, we have $|v_h^{(i)}-v_h^{(j)}| \geq 2-\mu$.
	Define
	\begin{align*}
	T := \bigcup_{h\in [\ell]\setminus L_i([\ell])}Q_{h}.
	\end{align*}
	We claim that
	$s\le \alpha(T)$.
	Indeed, if not, the graph $T[\{b^{(1)},\ldots, b^{(s)}\}]$ contains at least one edge $b^{(j)}b^{(j')}$, which lies in $Q_h$ for some $h\in [\ell]\setminus L_i([\ell])$. Thus $|v_{h}^{(j)} - v_{h}^{(j')}|\ge 2-\mu$, whence $h$ is a lengthy coordinate for $V_i\cap A$, a contradiction.	
	Therefore Observation~\ref{obs-indset} implies that
	\begin{align*}
	s\le \alpha\left(T\right)= 2^{\ell-(\ell-|L_i([\ell])|)}=2^{|L_i([\ell])|},
	\end{align*}
	finishing the proof of the lemma.
\end{proof}



To bound the clique number of $G$, we need one last lemma bounding the number of lengthy coordinates for $V_i\cap A$, for all $i\in[q]$.

\begin{lemma}\label{lem-allgoodcoord}
	$\sum_{i \in [q]}|L_i([\ell])| \leq \ell+p$.
\end{lemma}

\begin{proof}
	We claim that
	\begin{align*}
	\sum_{i \in [q]}|L_i([\ell])| &= \sum_{j \in [q]}\sum_{i \in [q]\setminus\{ j\}}|L_i([(c_{ij}-1)p+1,c_{ij}p])|+ \sum_{i \in [q]}|L_i([p(q-1)+1,\ell])|.
	\end{align*}
	Indeed, note that for each $i\in[q]$, $\bigcup_{j\in [q]\setminus\{i\}}[(c_{ij}-1)p+1,c_{ij}p]$ is a partition of $[p(q-1)]$, as $\chi$ is a proper $(q-1)$-edge-colouring of a $q$-clique. Thus the contribution to the left-hand side from $[p(q-1)]$ is $\sum_{i \in [q]}\sum_{j \in [q]\setminus\{ i\}}|L_i([(c_{ij}-1)p+1,c_{ij}p])|$. Swapping the summations proves the claim.
	
	Suppose for the sake of contradiction that at least one of the following holds:
	\begin{itemize}
		\item $\sum_{i \in [q]\setminus\{j\}}|L_i([(c_{ij}-1)p+1,c_{ij}p])| \geq p+1,\quad\text{for some } j \in [q]$;
		\item $\sum_{i \in [q]}|L_i([p(q-1)+1,\ell])| \geq \ell-p(q-1)+1$.
	\end{itemize}
	We claim that, in both cases, there are distinct $i,i' \in [q]$ and (not necessarily distinct) $h,h'\in[\ell]$ such that
	\begin{itemize}
		\item[(i)] $h\in L_i([\ell])$ and $h' \in L_{i'}([\ell])$;
		\item[(ii)] $(h,h')$ are $(i,i')$-related.
	\end{itemize}
	
	To see this, in the first case, by the Pigeonhole Principle, there are distinct $i,i' \in [q]\setminus\lbrace j\rbrace$ and $s \in [p]$ such that $h:=(c_{ij}-1)p+s$ and $h':=(c_{i'j}-1)p+s$ are lengthy for $V_i \cap A$ and $V_{i'} \cap A$ respectively. By definition, $(h,h')$ are $(i,i')$-related.
	
	In the second, again by the Pigeonhole Principle, some $h\in [p(q-1)+1,\ell]$ falls in $L_i([\ell])$ and $L_{i'}([\ell])$ for some distinct $i,i'$. It remains to recall that for any $h$ in this interval, $(h,h)$ is $(i,i')$-related for any distinct $i,i'$.
	
	Due to (i) above, there exist two pairs of vertices $\V{v}^{(1)},\V{v}^{(2)}\in V_{i}\cap A$ and $\V{u}^{(1)},\V{u}^{(2)}\in V_{i'}\cap A$ such that $|v^{(1)}_{h} - v^{(2)}_{h}|\ge 2-\mu$ and $|u^{(1)}_{h'} - u^{(2)}_{h'}|\ge 2-\mu$. At the same time, by (ii) and that $\V{u}^{(\ell)}\V{v}^{(\ell')} \in E(G)$ whenever $\ell,\ell' \in [2]$, we see that 
	$|v^{(\ell)}_h-u^{(\ell')}_{h'}| \leq \sqrt{2}-\mu$ for $\ell,\ell' \in [2]$.
	Therefore, we have four points $v^{(1)}_{h},v^{(2)}_{h},u^{(1)}_{h'},u^{(2)}_{h'}\in \sS^k(\mathbb{R})$ which contradict Theorem~\ref{thm-k4}.
	
	Thus the required sum is at most $pq+\ell-p(q-1) = \ell+p$.
\end{proof}

Let $x_i := |L_i([\ell])|$ for each $i \in [q]$.
By Lemmas~\ref{lem-cliquesize} and~\ref{lem-allgoodcoord}, we have that
\begin{align*}
|A|\le \sum_{i\in [q]}2^{x_i}\quad \text{subject to }x_1+\ldots+x_q \leq \ell +p \text{ and every }x_i \leq \ell.
\end{align*}
Optimising, we see that the maximum is attained by setting $x_i:=\ell$ and $x_{i'}:=p$ for some distinct $i,i' \in [q]$, and setting all others equal to $0$.
Thus the clique number of $G$ is
$$
\omega(G) \leq 2^\ell+2^p+q-2.
$$
This completes the proof of~\ref{it:MBE-clique} and hence of Theorem~\ref{thm-2tri-exact}.

\section{Upper bounds}\label{sec-UB}

In this section we will show that one can find upper bounds for $\varrho_p(q)$ by considering a cleaner problem on weighted graphs. First, in Section~\ref{sec-reduction-weighted-graph}, we introduce a family of weighted graphs $\widetilde{\cG}_p(q)$ and show that a certain weighted Tur\'an-type density $\pi(\widetilde{\cG}_p(q))$ bounds $\varrho_p(q)$ from above; see Lemma~\ref{reduce}. In order to bound $\pi(\widetilde{\cG}_p(q))$, we prove in Section~\ref{sec-embedding} an embedding lemma to study the structure of the edge weights of weighted graphs in a simpler subfamily $\cG_p(q) \subseteq \widetilde{\cG}_p(q)$; see Lemma~\ref{embedding}. From this structural information, Theorem~\ref{thm-upper-easy} then follows fairly easily; see Section~\ref{sec-upper-easy}.
The bulk of the work then is devoted to deriving the upper bound for $\pi(\widetilde{\cG}_p(q))$ when $p=\{3,4\}$; see Section~\ref{sec-upper-main}.


\subsection{Reduction to weighted graphs}\label{sec-reduction-weighted-graph}

Given $p\in\mathbb{N}$, a \emph{$p$-weighted graph} is a pair $G=(V,w)$ consisting of a vertex set $V$
and a symmetric function $w : V^2 \rightarrow \lbrace 0,1,\ldots, p\rbrace$ such that
$w(x,x)=0$ for all $x \in V$.
Given $x \in V$, let $d_G(x) := \sum_{y \in V-x}w(x,y)$ and $\delta(G) := \min_{x \in V}d_G(x)$. A $p$-weighted graph is \emph{positive} if all its pairs $x \neq y$ receive strictly positive weights.
Given a family $\mathcal{G}$ of $p$-weighted graphs and a $p$-weighted graph $G=(V,w)$, we say that $G$ is \emph{$\mathcal{G}$-free} if there is no $U \subseteq V$ such that $G[U] \in \mathcal{G}$, where
$G[U] := (U,w|_{U^2})$.

We will use Szemer\'edi's regularity lemma so need to define the associated notions.
Given $\eps>0$, a bipartite graph with vertex bipartition $A,B$ (or the pair $(A,B)$) is said to be \emph{$\eps$-regular} if, for all $A' \subseteq A$ and $B' \subseteq B$ with $|A'| \geq \eps|A|$ and $|B'| \geq \eps|B|$, we have $|d_G(A',B')-d_G(A,B)| \leq \eps$.
If additionally $d_G(A,B) \geq d$, then we say that $G$ is \emph{$(\eps,\geq\!d)$-regular}.

\begin{theorem}[Regularity lemma]\label{reg}
	For every $\eps>0$ and integer $M'$ there exist integers $M,n_0$ such that if $G$ is a graph on $n \geq n_0$ vertices, then there is a partition of $V(G)$ into $V_0,V_1,\ldots,V_m$ for some $M' \leq m \leq M$ so that $|V_0| \leq \eps n$; $|V_1|=\ldots=|V_m|=:n'$ and for each $i\in[m]$, $G[V_i,V_j]$ is $\eps$-regular for all but at most $\eps m$ pairs $(i,j)$.
\end{theorem}

We first define a family $\widetilde{\cG}_p(q)$ of $p$-weighted graphs which arise from regularity partitions.
Informally, $\widetilde{\cG}_p(q)$ contains all positive $p$-weighted graphs such that any of their pseudorandom blow-ups with sublinear $p$-independence number contain a $K_q$. 

\begin{definition}\label{def:Gpq}
	~
	
	\vspace{-0.3cm}
	\begin{itemize}
		\item Given a positive $p$-weighted graph $G=(V,w)$ and $\gamma,\zeta,\eta>0$, we say that an extension of $w$ to $V^2 \cup V$ (i.e.~to also include a vertex weighting $\lbrace w(v) : v \in V\rbrace$) taking values in $\lbrace 1,\ldots,p\rbrace$ is \emph{valid wrt $\gamma,\zeta,\eta$} if there exists $n_0(\gamma,\zeta,\eta)>0$ such that the following holds for all integers $n \geq n_0$:
		\begin{itemize}
			\item[] 
			Let $H = (W,E)$ be an $n$-vertex graph such that
			there is a vertex partition $W=\bigcup_{v \in V}W_v$ with $|W_v| \geq \eta n$ and $\alpha_p(H[W_v])\le \gamma|W_v|$ for all $v \in V$, and
			$H[W_u,W_v]$ is $(\zeta,\geq\!\frac{w(u,v)-1}{p}+\eta)$-regular for all $uv\in{V\choose 2}$.
			Then $H$ contains a clique of size $\sum_{v \in V}w(v)$.
		\end{itemize}
		
		\item Let $\widetilde{\mathcal{G}}_p(q)$ be the class of positive $p$-weighted graphs $(V,w)$ with $|V| \leq q$ such that for all $\eta>0$ there exist $\zeta,\gamma>0$ and a valid vertex weighting wrt $\gamma,\zeta,\eta$ with $\sum_{v \in V}w(v) \geq q$.
		\item Given positive integers $p \leq q$, let
		\begin{align*}
			\pi(\widetilde{\cG}_p(q)):=\sup\{d \in [0,1]:~&\text{ every sufficiently large } p\text{-weighted }G \text{ with }\\
			&\delta(G) > dp|V| \text{ is } \widetilde{\cG}_p(q)\text{-free} \}.
		\end{align*} 
	\end{itemize}
\end{definition}
For example, it is not hard to see that any positive $p$-weighted graph on two vertices is in $\widetilde{\cG}_p(p+1)$. 
Indeed, let $H$ be an $n$-vertex graph as in the definition, so $H$ is a regular pair $(A,B)$. Choose a typical vertex $v \in A$; by positivity, the density of $H$ is at least $\eta$, so $d_G(v,B) \geq \eta|B| > \gamma|B|$.
Now there is a copy of $K_p$ in the neighbourhood of $v$ in $B$, so $K_{p+1} \subseteq H$.
Note that $\widetilde{\cG}_p(q)$ is a finite family, so $\gamma,\zeta$ may be chosen uniformly.


The goal of this section is to prove the following  lemma, which allows us to upper bound $\varrho_p(q)$ by $\pi(\widetilde{\cG}_p(q))$. 
This lemma is related to several theorems in~\cite{EHSSSz} and its proof follows the same approach, using Szemer\'edi's regularity lemma.

\begin{lemma}\label{reduce}
	For all positive integers $p \leq q$ we have $\varrho_p(q) \leq \pi(\widetilde{\cG}_p(q))$.
	
	\medskip
	In other words:
	For all $\delta>0$ and $p \leq q \in \mathbb{N}$, let $d\ge \pi(\widetilde{\cG}_p(q))\in [0,1]$. That is, for every $q \geq p$,
	every sufficiently large $p$-weighted graph $G=(V,w)$ with $\delta(G) > dp|V|$ has a subset $U \subseteq V$ such that $G[U]$ lies in $\widetilde{\cG}_p(q)$. 
	Then there exists $\eps>0$ such that whenever $n_0=n_0(\eps)$ is sufficiently large,
	every graph $H$ on $n \geq n_0$ vertices with $e(H) \geq (d+\delta)\binom{n}{2}$ edges and $\alpha_p(H) \leq \eps n$ contains a copy of $K_q$.
\end{lemma}


\begin{proof}
	Let $\delta>0$ and $p \leq q \in \mathbb{N}$, and $d\ge \pi(\widetilde{\cG}_p(q))$. Let $\eta$ be such that $0 < \eta \ll \delta \ll 1/q$, where we have decreased $\delta$ if necessary (doing so will only prove a stronger result). 
	Let $\zeta,\gamma>0$ be such that every $G \in \widetilde{\cG}_p(q)$ has a valid vertex weighting wrt $\gamma,\zeta,\eta$ with $\sum_{v \in V}w(v) \geq q$.
	Since a vertex weighting which is valid wrt $\gamma,\zeta,\eta$ is also valid wrt $\gamma',\zeta',\eta$ whenever $\gamma' \leq \gamma$ and $\zeta' \leq \zeta$, by decreasing $\gamma,\zeta$ if necessary, we may assume that $0 < \gamma \ll \zeta \ll \eta$.
	Choose an additional parameter $M'$ such that
	$0 < \gamma \ll 1/M' \ll \zeta \ll \eta \ll \delta$, and any $p$-weighted graph on at least $M'$ vertices is sufficiently large in the sense of the definition of $\pi(\widetilde{\cG}_p(q))$.
	Apply Theorem~\ref{reg} (the Regularity Lemma) to $\zeta,M'$ to obtain $M,n_1$.
	By increasing $M,n_1$ and decreasing $\gamma$ if necessary, we may assume that $1/n_1 \ll \gamma \ll 1/M \ll 1/M'$
	and that $n_1 \geq (n_0(\gamma,\zeta,\eta))^2$ from Definition~\ref{def:Gpq}.
	Altogether we have
	$$
	0 < 1/n_1 \ll \gamma \ll 1/M \ll 1/M' \ll \zeta \ll \eta \ll \delta \ll 1/q \leq 1/p.
	$$
	The choice of $d$ implies that for every sufficiently large $p$-weighted graph $G=(V,w)$ with $\delta(G)>dp|V|$, there exists $U \subseteq V$ such that $G[U]$ is positive and a valid vertex weighting wrt $\gamma,\zeta,\eta$ of $U$ with $\sum_{u \in U}w(u) \geq q$.
	Let $H'$ be a graph on $n' \geq n_1$ vertices with $e(H') \geq (d+\delta)\binom{n'}{2}$ and  $\alpha_p(H')\le \gamma^2 n'$. To prove the lemma, we will show that $H' \supseteq K_{q}$ (so $\eps:=\gamma^2$). Using a standard trick of repeatedly removing low degree vertices, we can pass from $H'$ to an $n$-vertex subgraph $H$ with $n\ge \de^{1/4}n'\ge n_0(\gamma,\zeta,\eta)$ and $\de(H)\ge (d+\frac{\de}{2})n$.
	
	Apply the Regularity Lemma to $H$ with parameters $\zeta,M'$
	to obtain a partition $V_0 \cup V_1 \cup \ldots \cup V_m$ of its vertex set where $M' \leq m \leq M$ satisfying the conclusions of Theorem~\ref{reg}.
	Let $A = (a_{ii'})$ be the symmetric $m\times m$ matrix in which
	$$
	a_{ii'} := 
	\begin{cases}
		\lfloor p(d_H(V_i,V_{i'})-\eta)\rfloor+1, \quad\mbox{if } H[V_i,V_{i'}] \text{ is } (\zeta,\geq\!\frac{\delta}{4})\text{-regular};\\
		0, \quad\mbox{ otherwise.} 
	\end{cases}
	$$
	So $A$ has entries in $\{0,\ldots,p\}$.
	Let $G=(V,w)$ be the $p$-weighted graph with $V = \{v_1,\ldots,v_m\}$ and $w(v_i,v_{i'}) := a_{ii'}$.
	Write $d_{ii'} := d_H(V_i,V_{i'})$.
	Note that if $a_{ii'}$ is positive, then $a_{ii'}$ equals $k+1$ if and only if $\frac{k}{p}+\eta \leq d_{ii'} < \frac{k+1}{p}+\eta$.
	Thus $a_{ii'}\geq p(d_{ii'}-\eta) \geq (1-\sqrt{\eta})pd_{ii'}$, since $d_{ii'} \geq \frac{\delta}{4}$.
	Standard results on the `reduced graph' of $H$ imply that for every $i \in [m]$, the sum of $d_{ii'}$ over all $V_{i'}$ such that $(V_i,V_{i'})$ is $(\zeta,\geq\!\frac{\delta}{4})$-regular is at least $\delta(H)\cdot\frac{m}{n}-\frac{\delta}{4}\cdot  m$.
	Thus for all $v_i \in V$,
	\begin{align}\label{eq:gA}
		d_G(v_i) = \sum_{i' \in [m]\setminus\{i\}}a_{ii'} \geq (1-\sqrt{\eta})\sum_{i' \in [m]\setminus\{i\}}pd_{ii'} \geq (1-\sqrt{\eta})pm\left(d+\frac{\delta}{4}\right) \geq p\left(d+\frac{\delta}{5}\right)m.
	\end{align}
	So by our choice of $d$, there exists $U \subseteq V$ such that $G[U]$ is positive and a valid vertex weighting wrt $\gamma,\zeta,\eta$ of $U$ with $\sum_{v \in U}w(v) \geq q$.
	Let $I \subseteq [m]$ be such that $U = \{v_i: i \in I\}$ and let $H_I$ be the subgraph of $H$ induced by $\{V_i: i \in I\}$.
	Then for all $i \in I$ we have $\alpha_p(H_I[V_i])\le \alpha_p(H') \leq \gamma^2 n'\le  (\gamma^2 \de^{-1/4}\cdot 2M)\frac{n}{2M}\leq  \gamma|V_i|$.
	Moreover, for distinct $i,i' \in I$, $H_I[V_i,V_{i'}] = H[V_i,V_{i'}]$ is $(\zeta,\geq\!\frac{\delta}{4})$-regular (since $a_{ii'}>0$).
	Also, $w(v_i,v_{i'})=a_{ii'}=\lfloor p(d_{ii'}-\eta)\rfloor + 1$ so $d_{ii'}\geq \frac{w(v_i,v_{i'})-1}{p}+\eta$.
	By the definition of a valid vertex weighting, $K_q \subseteq H_I\subseteq H'$.
	
	Thus $\rt_p(n,K_q,\gamma^2 n') \leq (d+\delta)\binom{n'}{2}$ for all $n' \geq n_1$ and it follows that $\varrho_p(q) \leq \pi(\widetilde{\cG}_p(q))$.
\end{proof}

\subsection{An embedding lemma via dominating extensions}\label{sec-embedding}
To upper bound $\pi(\widetilde{\cG}_p(q))$, we will consider only valid vertex weightings with a particular property.
\begin{definition}
	Let $G=(V,w)$ be a positive $p$-weighted graph and $\lbrace v_1,\ldots,v_m\rbrace$ be an enumeration of $V$. An extension of $w$ to $V^2 \cup V$ is \emph{dominating} if  for all $j \in \{2,\ldots,m\}$, writing $a:=w(v_j)$, the multiset of backwards edge weights $\lbrace w(v_i,v_j) : i \in [j-1]\rbrace$ dominates
	$$
	\Big\lbrace \underbrace{\textstyle{\frac{p(a-1)}{a}+1,\ldots,\frac{p(a-1)}{a}+1}}_{j-2}, a\Big\rbrace
	$$
	as ordered multisets.
	The \emph{size} of an extension of $w$ to $V^2 \cup V$ is $\sum_{v \in V}w(v)$. 
	Let $\mathcal{G}_p(q)$ be the set of positive $p$-weighted graphs $G=(V,w)$ with dominating extension of size at least $q$.
\end{definition}

For example, $\{3,4,4\}$ dominates $\{3,3,4\}$ but not $\{2,2,5\}$.
There is no constraint on the weight of $v_1$, so we can always choose $w(v_1)=p$.
Note that, writing $t := \max_{uv \in \binom{V}{2}}w(u,v)$, we have that $\{p,t,1,\ldots, 1\}$ is dominating.
Indeed, enumerate $V$ so that $w(v_1,v_2)=t$ is maximal and write $a_j := w(v_j)$ for all $j \in [m]$.
The multisets of backwards edge weights are $\{t\},\{w(v_1,v_3),w(v_2,v_3)\},\{w(v_1,v_4),w(v_2,v_4),w(v_3,v_4)\}\ldots$, which, respectively,
dominate
$$
\{a_2\},\left\{\frac{p(a_3-1)}{a_3}+1, a_3\right\},\left\{\frac{p(a_4-1)}{a_4}+1,\frac{p(a_4-1)}{a_4}+1,a_4\right\},\ldots = \{t\},\{1,1\},\{1,1,1\},\ldots,
$$
as required. 
Thus we have
\begin{align}\label{eq:mq}
	&G \in \cG_p(p+t+m-2), \quad\forall\text{ positive } p\text{-weighted } m\text{-vertex } G \text{ with an edge of weight}\geq t.
\end{align}
We consider dominating extensions due to the following averaging claim.
\begin{claim}\label{cl:int}
	Let $a \leq p$ be positive integers, let $\eta > 0$ and let $Y$ be a set.
	Given sets $A_1,\ldots,A_p \subseteq Y$ with $|A_i| \geq (\frac{a-1}{p}+\eta)|Y|$, there is some $I \subseteq [p]$ with $|I|=a$ such that $|\bigcap_{i \in I}A_i| \geq p^{-a}\eta|Y|$.
\end{claim}

This will imply that among $p$ typical vertices in one part $A$ of a regular pair $(A,B)$ of density at least $\frac{a-1}{p}+\eta$, there are $a$ of them which share a large common neighbourhood in $B$.
We will use this to extend a clique by $a$ vertices in a new regularity cluster.

\begin{poc}
	Note that the following lower bound for $a$-wise intersections holds:
	$$
	\textstyle{(p-(a-1))\left|\bigcup_{I \subseteq [p]:|I|=a}\bigcap_{i \in I}A_i\right| \geq \sum_{i \in [p]}|A_i|-(a-1)|\bigcup_{i \in [p]}A_i|}.
	$$
	Indeed, the left-hand side counts every element in an $\ell$-wise intersection $0$ times for $\ell\leq a-1$ and $p-(a-1)$ times for $\ell \geq a$, while the right-hand side counts every element in an $\ell$-wise intersection $\ell-(a-1)$ times.
	As the right-hand side is at least $p(\frac{a-1}{p}+\eta)|Y|-(a-1)|Y|=p\eta|Y|$, we see that there is some $I \subseteq [p]$ with $|I|=a$ such that $|\bigcap_{i \in I}A_i| \geq \binom{p}{a}^{-1}(p-(a-1))^{-1}p\eta|Y| \geq p^{-a}\eta|Y|$.
\end{poc}

\begin{lemma}\label{embedding}
	For all positive integers $p \leq q$, $\mathcal{G}_p(q) \subseteq \widetilde{\mathcal{G}}_p(q)$.
	
	\medskip
	This is a consequence of the following statement:
	Let $p,m$ be integers and let $G=(V,w)$ be a positive $p$-weighted graph with enumeration $V := \lbrace v_1,\ldots,v_m\rbrace$
	and let $0 < \gamma \ll \zeta \ll \eta \ll 1/m$. 
	Then any dominating extension of $w$ is valid wrt $\gamma,\zeta,\eta$. 
\end{lemma}

\begin{proof}
	To see why the first statement follows from the second, let $p \leq q$ be positive integers and
	let $G \in \cG_p(q)$. So $G$ is a positive $p$-weighted graph with a dominating extension of size at least $q$.
	Let $0 < \gamma \ll \zeta \ll \eta \ll 1/q,1/m$, where $|V|=m$.
	By the second statement, the dominating extension is valid wrt $\gamma,\zeta,\eta$. So $G \in \widetilde{\cG}_p(q)$.
	
	It remains to prove the second statement.
	Let $0 < \gamma \ll \zeta \ll \eta \ll 1/m$ and let $H = (W,E)$ be an $n$-vertex graph such that
	there is a vertex partition $W=\bigcup_{v \in V}W_v$ with $|W_v| \geq \eta n$ and $\alpha_p(H[W_v])\le \gamma|W_v|$ for all $v \in V$, and
	$H[W_u,W_v]$ is $(\zeta,\geq\!\frac{w(u,v)-1}{p}+\eta)$-regular for all $uv\in{V\choose 2}$.
	Suppose that $w$ has a dominating extension.
	To prove the lemma, we need to find $K_q \subseteq H$ where $q := \sum_{v \in V}w(v)$.

	We will find a clique that contains $w(v_i)$ vertices in $W_i:=W_{v_i}$, in the reverse order $i=m,m-1,\ldots,1$.
	Suppose for some $1 \leq r \leq m$ that for every $j > r$ we have found vertices $x_1^{j},\ldots,x_{w(v_j)}^j$ in $W_j$
	such that, writing $X_{r+1}$ for their union, we have $H[X_{r+1}]$ is a clique and, for all $i\in [r]$, $W_i^{r+1} := \bigcap_{x \in X_{r+1}}N_H(x,W_i)$ satisfies $|W_i^{r+1}| \geq (\eta/4p^p)^{m-r}|W_i|$.
	We will extend this clique by adding $w(v_r)$ vertices from $W_r^{r+1}$.
	Note that $m-r \leq |X_{r+1}| < q$.
	
	For each $i \in [r-1]$, let $d_{ir} := \frac{w(v_i,v_r)-1}{p}$.
	By standard results (the Slicing Lemma for regular pairs), the pair $(W_r^{r+1},W_i^{r+1})$ is $(\zeta^{2/3},\geq\!d_{ir}+\eta/2)$-regular for all $i \in [r-1]$, as $\zeta \ll \eta,1/q,1/p$.
	By further standard results (on superregular pairs), for each $i \in [r-1]$, there is $W_{r,i} \subseteq W_r^{r+1}$   
	with $|W_{r,i}| \geq (1-\sqrt{\zeta})|W_r^{r+1}|$ such that
	$$
	|N_H(x,W_i^{r+1})| \geq (d_{ir}+\eta/3)|W_i^{r+1}| 
	$$
	for every $x \in W_{r,i}$.
	Letting $W_r^*=\cap_{i\in[r-1]}W_{r,i} \subseteq W_r^{r+1}$, we have  $|W_r^*| \geq (1-r\sqrt{\zeta})|W_r^{r+1}|$. 
	Next, let $a:= w(v_r)$. 
	Since $w$ is a dominating extension, there is $s \in [r-1]$ such that $w(v_s,v_r) \geq a$, whence $d_{sr} \geq \frac{a-1}{p}$, and, for all $i \in [r-1]\setminus\lbrace s\rbrace$, $w(v_i,v_r) \geq \frac{p(a-1)}{a}+1$, whence $d_{ir}\ge \frac{a-1}{a}$.
	
	Now, since $\alpha_p(H) \leq \gamma n \leq (1-r\sqrt{\zeta})(\eta/4p^p)^{m-r}\eta n \leq |W^*_r|$, $H$ induces a copy of $K_{p}$ on some $Q\subseteq W_r^*$. Since each $x\in Q$ is adjacent to at least $d_{sr}+\eta/3$ proportion of the vertices in $W_s^{r+1}$ and $d_{sr}=\frac{a-1}{p}$, Claim~\ref{cl:int} yields an $a$-subset $I\subseteq Q$ such that, letting $W_i^r := \bigcap_{x \in I}N_H(x,W_i^{r+1})$ for all $i \in [r-1]$, we have 
	$$|W_s^r| \geq p^{-a}\cdot\eta/3\cdot|W_s^{r+1}| \geq (\eta/4p^p)^{m-r+1}|W_s|.$$
	Let $I =: \lbrace x^r_1,\ldots,x^r_{a}\rbrace$ and $X_r := X_{r+1}\cup I$.
	By construction, $H[X_r]$ is a clique.
	Recall that for all $i \in [r-1]\setminus\lbrace s\rbrace$, we have $d_{ir} \geq (a-1)/a$,
	thus by inclusion-exclusion, 
	$$|W_i^r| \geq a(d_{ir}+\eta/3)|W_i^{r+1}| - (a-1)|W_i^{r+1}| \geq \eta|W_i^{r+1}|/4 \geq (\eta/4p^p)^{m-r+1}|W_i|.$$
	
	Therefore we can complete the embedding sequentially to obtain a vertex set $X_1$ of size $q=\sum_{i=1}^{p}w(v_i)$ upon which $H$ spans a clique.
\end{proof}


\begin{remark}
	It is worth noting that dominating extensions are not always the best ones to take. 
	Consider for example integers $p,s,t$ with $2s-1\le p \leq s(s-1)$ and let $G=(\lbrace v_1,\ldots,v_t\rbrace,w)$, where $w(v_i,v_j)=s$ for all $ij\in{[t]\choose 2}$. As $\frac{p}{2}+1>s$, in any dominating extension, at most two vertices can have weight at least $2$, offering at best $(w(v_1),\ldots,w(v_t)) =(p,s,1,\ldots,1)$. But in fact there is a valid vertex weighting of larger size, namely $(p,s,2,1,\ldots,1)$ and so $G\in \widetilde{\cG}_p(p+s+t-1)$. Indeed, as in the proof of Lemma~\ref{embedding}, we can put one vertex in each of $W_t,\ldots,W_4$, whose common neighbourhood $W_i^4$ in each $W_i$ with $i \in [3]$ is linear. 
	As the $W_i$'s are pairwise $(\zeta,\geq\!\frac{s-1}{p}+\eta)$-regular, Claim~\ref{cl:int} implies that among the vertices of a $K_{p}$ in $W_3^4$, there are $s$ with linear common neighbourhood in $W_2^4$. It suffices to find an edge in this $K_s$ whose common neighbourhood in $W_1^4$ is linear. Since $(\frac{s-1}{p}+\eta)s>1$, averaging (or Claim~\ref{cl:int} again) yields such an edge.
\end{remark}

To prove Theorems~\ref{thm-upper-main} and~\ref{thm-upper-easy} we will consider weighted graphs (and will not require anything to do with regularity). Indeed, Lemmas~\ref{reduce} and~\ref{embedding} imply the following.

\begin{lemma}\label{lem:summary}
	Let $p \leq q$ be positive integers.
	Suppose that for all $p$-weighted graphs $G=(V,w)$ with $\delta(G) > dp|V|$, there is $J \subseteq V$ such that $G[J] \in \cG_p(q)$. Then $\varrho_p(q) \leq d$.\qed
\end{lemma}

\subsection{Proof of Theorem~\ref{thm-upper-easy}}\label{sec-upper-easy}

Given an $m\times m$ matrix $A$, define
$$
g(A) := \textstyle{\max\left\lbrace \bm{u}^\intercal A\bm{u}: \bm{u} = (u_1,\ldots,u_m)^\intercal;~ \sum_{i \in [m]}u_i=1;~ u_i \geq 0\right\rbrace}
$$
(where the maximum is attained since it is taken over a compact set). Say that any $\bm{u}$ which attains the maximum is \emph{optimal} for $A$.
We say that $A$ is \emph{dense} if $A$ has zero diagonal and, for any $i \in [m]$, the submatrix $A'$ obtained by deleting the $i$-th row and $i$-th column
satisfies $g(A') < g(A)$. The following properties of dense matrices and their optimal vectors will be useful.

\begin{lemma}\label{sym}
	Let $m \in \mathbb{N}$ and let $A = (a_{ij})$ be a dense symmetric $m\times m$ matrix with entries in $\lbrace 0,1,\ldots,p\rbrace$ and let $\bm{u}$ be optimal for $A$.
	Then
	\begin{itemize}
		\item[(i)] $A$ is positive, that is, $a_{ij}>0$ for all $1\le i<j\le m$;
		\item[(ii)] $u_i>0$ for all $i \in [m]$;
		\item[(iii)] $\sum_{i \in [m]\setminus\lbrace j \rbrace}a_{ij}u_i = g(A)$ for all $j \in [m]$.
	\end{itemize}
\end{lemma}

\begin{proof}
	Part~(i) is Lemma~3.3 in~\cite{EHSSSz} and follows from a version of Zykov's symmetrisation and that $A$ is dense. For~(ii), one can easily see that every $u_i$ is positive (otherwise the matrix $A'$ obtained by deleting the $i$-th row and $i$-th column of $A$ satisfies $g(A')=g(A)$).
	Part~(iii) follows from~(ii) and the method of Lagrange multipliers (and the fact that $A$ has zero diagonal).
\end{proof}

The following lemma together with Lemma~\ref{lem:summary} implies Theorem~\ref{thm-upper-easy}.

\begin{lemma}\label{uppertilde}
	Let $p,s,t$ be positive integers satisfying $t(t-2) \leq s \leq t^2$ and $s+t-1 \leq p$. Then
	for every $p$-weighted $n$-vertex graph $G=(V,w)$ with $\delta(G) > \frac{s(t-1)}{t}\cdot n$,
	there is $J \subseteq V$ such that $G[J] \in \cG_p(p+s+t-1)$.
\end{lemma}

\begin{proof}
	Write $q := p+s+t-1$.
	Let $G=(V,w)$ be a $p$-weighted graph on $n$ vertices such that $\delta(G) > \frac{s(t-1)}{t} \cdot n$. 
	Let $V=\{v_1,\ldots,v_n\}$ be an enumeration and let $A = (a_{ij})$ be the symmetric $m \times m$ matrix with $a_{ij}=w(v_i,v_j)$ for $i \neq j$ and $0$ otherwise.
	Choose $J \subseteq [m]$ such that the submatrix $A'$ obtained by retaining the rows and columns of $A$ with indices in $J$ satisfies $g(A') \geq g(A)$, and $|J|$ is minimal.
	Then $A'$ is dense (and non-empty since $g$ is $0$ on the empty matrix).
	Lemma~\ref{sym}(i) implies that $A'$ is positive.
	Let $m := |J|$
	and let $\bm{u}$ be optimal for $A'$ (so $\bm{u}$ has length $m$).
	Writing $\bm{u}_n=(\frac{1}{n},\ldots,\frac{1}{n})^\intercal$ of length $n$, Lemma~\ref{sym}(iii) implies that for all $j\in[m]$,
	$$
	\sum_{i \in J\setminus\{j\}}w(v_i,v_j)u_i = g(A') \geq g(A) \geq \bm{u}_n^\intercal A \bm{u}_n = \frac{1}{n^2}\sum_{ij \in \binom{n}{2}}2w(v_i,v_j) \geq \frac{1}{n}\delta(G) > \frac{s(t-1)}{t}.
	$$
	Let $G' := G[J]$, so $G$ is a positive $p$-weighted graph on $m\ge 2$ vertices.
	
	If $m \geq q-p+1$ or there are distinct $i,j \in J$ with $w(v_i,v_j) \geq q-p-m+2$, then by~(\ref{eq:mq}), $G' \in \mathcal{G}_p(q)$ (note $q-p-m+2 \leq p$).
	
	Thus we may assume that $m \leq q-p$ and $w(v_i,v_j) \leq q-p-m+1=s+t-m$ for all distinct $i,j \in J$.
	Let $i \in J$ be such that $u_i \geq u_j$ for all $i \in J$. So $u_i \geq \frac{1}{m}$ and
	\begin{equation}\label{eq:cont}
		\frac{s(t-1)}{t} < \sum_{i \in J\setminus\{j\}}w(v_i,v_j)u_i  \leq (s+t-m)\cdot \frac{(m-1)}{m}.
	\end{equation}
Multiplying by $m$, we have $(m-t)(m-\frac{s+t}{t}) < 0$, which by $t-1 \leq \frac{s+t}{t} \leq t+1$ and $m \in \mathbb{N}$ is a contradiction.
\end{proof}

\subsection{Proof of Theorem~\ref{thm-upper-main}}\label{sec-upper-main}
Throughout this section, we always assume $p\in\{3,4\}$. 
Given a $p$-weighted graph $G=(V,w)$ and a $J \subseteq V$ with $J = \{v_1,\ldots,v_m\}$,
the ``maximal'' dominating extension $w$ of $G[J]$ is, by definition, such that,
for each $j\in [m]$, we have
\begin{itemize}
	\item $w(v_j) = p$ if and only if $w(v_i,v_j) = p$ for all $i \in [j-1]$;
	\item $w(v_j) \geq a$, for every $2\le a\le p-1$, if and only if $w(v_i,v_j)\geq a$ for all $i \in [j-1]$ with equality \emph{at most once};
	\item $w(v_j)\geq 1$ if and only if $w(v_i,v_j) \geq 1$ for all $i \in [j-1]$.
\end{itemize}

We will find it convenient to write $\tilde{w}(x,y) := p - w(x,y)$; and given $K \subseteq V$, to let
$$
\tilde{w}(K) := \sum_{xy \in \binom{K}{2}}\tilde{w}(x,y)\quad\text{and}\quad \gamma_K(x) := \sum_{y \in K\setminus\{x\}}\tilde{w}(x,y)\quad\text{for all }x \in V.
$$

\begin{proposition}\label{prop-hero}
	For every $p$-weighted graph $G=(V,w)$, there exists a set $K \subseteq V$ such that
	\begin{itemize}
	\item[(i)] $G[K] \in \cG_p(p|K|-\tilde{w}(K))$;
	\item[(ii)] we have $\gamma_K(y) \leq p-1$ for all $y \in K$ and $\gamma_K(x) \geq p$ for all $x \in V\setminus K$;
	\item[(iii)] if $x \in V\setminus K$ and $y \in K$, we have $\gamma_{K\setminus\{y\}}(x) \geq \gamma_K(y)$.
	\end{itemize}
\end{proposition}

\begin{proof}
We say a non-empty set $K$ is \emph{heroic} if, for all $\varnothing \neq L \subseteq K$ we have $G[L] \in \mathcal{G}_p(p|L|-\tilde{w}(L))$. Observe that every singleton in $V$ is heroic, as it can be given weight $p$, and subsets of heroic sets are heroic.

\begin{claim}\label{cl-hero}
If $K$ is heroic and $x \in V\setminus K$ with $\gamma_K(x) \leq p-1$, then $K \cup \{x\}$ is heroic.
\end{claim}
\begin{poc}
	As any singleton is heroic, it suffices to check that for any $\varnothing\neq L\subseteq K$, $G[L\cup \{x\}]\in\cG_p(p|L \cup \lbrace x\rbrace| - \tilde{w}(L \cup \lbrace x\rbrace))$. To see this, add $x$ to the end of the enumeration of $L$ with dominating extension of maximal size. We claim that setting $w(x):=p-\gamma_L(x)$ extends it to a dominating extension of $L\cup \{x\}$. Indeed, for each $v\in L$, \begin{equation*}
		w(v,x)=p- \tilde{w}(v,x)= p-\gamma_L(x)+\sum_{u\in L\setminus \{v\}}\tilde{w}(u,x)\geq p-\gamma_L(x),
	\end{equation*}
	where equality holds only when $w(u,x)=p$ for each $u\in L\setminus \{v\}$.
	We need to check the dominating properties.
	Note first that every $w(v,x) \geq 1$ (as $\gamma_L(x) \leq p-1$), and so we may assume $w(x)\ge 2$, i.e.~$\gamma_L(x)\le p-2$. 
	If $1 \leq \gamma_L(x) \leq p-2$, $2 \leq w(x) \leq p-1$ and $w(v,x) \geq w(x)$ with equality at most once, since if $w(v,x)=w(x)$ then $w(u,x)=p>w(x)$ for all $u \in L\setminus\{v\}$.
	If $\gamma_L(x)=0$, then $w(x)=p$ and $w(v,x)=p$ for all $v \in L$.
	So the extension to $L \cup \{x\}$ is dominating.
	As $L$ is heroic, we have
	$$
	\sum_{v \in L\cup\lbrace x\rbrace}w(v) \geq p|L|-\tilde{w}(L)+p-\gamma_L(x) = p|L \cup \lbrace v\rbrace| - \tilde{w}(L \cup \lbrace v\rbrace),
	$$
	as required.
\end{poc}

A heroic set $K$ is \emph{herculean} if
\begin{itemize}
	\item $p|K|-\tilde{w}(K)$ is maximal; 
	\item subject to the above, $|K|$ is minimal.
\end{itemize}
We claim that we can take any herculean $K$ for the required set.
Indeed,~$K$ satisfies (i).
	
	Now we prove (ii). Let $K' := K\setminus\lbrace y \rbrace$ and suppose $\gamma_K(y) \geq p$. Note that $K'\neq\varnothing$, since otherwise $\gamma_K(y)=0$. Then, using that $K$ is herculean, $p|K'|-\tilde{w}(K')<p|K|-\tilde{w}(K)$, entailing
	$\gamma_K(y)=\tilde{w}(K)-\tilde{w}(K')<p$, a contradiction.
	
	Suppose instead there is $x \in V\setminus K$ with $\gamma_K(x) \leq p-1$.
	Then Claim~\ref{cl-hero} implies that $K \cup\lbrace x\rbrace$ is heroic with 
	$$p|K \cup \lbrace x\rbrace| - \tilde{w}(K \cup \lbrace x\rbrace)=p|K|-\tilde{w}(K)+p-\gamma_K(x)>p|K|-\tilde{w}(K),$$ 
	contradicting the fact that $K$ is herculean.
	
	For~(iii), let $x \in V\setminus K$ and $y \in K$.
	Suppose that $\gamma_{K\setminus\{y\}}(x) < \gamma_K(y)$.
	Then~(ii) implies that $\gamma_{K\setminus\{y\}}(x) \leq p-1$.
	As $K\setminus\lbrace y \rbrace$ is heroic (as a subset of a heroic set), Claim~\ref{cl-hero} implies that $K' := (K\setminus\lbrace y \rbrace)\cup\lbrace x\rbrace$ is heroic. But $|K'|=|K|$ and 
	$$p|K'|-\tilde{w}(K') = p|K'|-\tilde{w}(K)+\gamma_K(y)-\gamma_{K\setminus\{y\}}(x) > p|K|-\tilde{w}(K),$$ 
	a contradiction to $K$ being herculean. 
\end{proof}

We are now ready to prove the final upper bound.
Recall that
$$
\varrho^*_p(pt+2) = \frac{(t-1)(2p-1)+1}{t(2p-1)+1} =
\begin{cases}
	\frac{5t-4}{5t+1} &\text{if }p=3,\\
	\frac{7t-6}{7t+1} &\text{if }p=4.\\
\end{cases}
$$
The lower bounds in Theorem~\ref{thm-upper-main} follow from Corollary~\ref{cor}; for the upper bounds, using Lemma~\ref{lem:summary}, it suffices to prove the following lemma.

\begin{lemma}\label{smallp}
	Let $p \in\lbrace 3,4\rbrace$ and let $t\in\mathbb{N}$.
	Let $G=(V,w)$ be a $p$-weighted $n$-vertex graph with
	$$\delta(G)>p \cdot \varrho_p^*(pt+2) \cdot n.$$
	Then there is $J \subseteq V$ such that $G[J] \in \cG_p(pt+2)$.
\end{lemma}

\begin{proof}
	Let $G=(V,w)$ be an $n$-vertex $p$-weighted graph with $\delta(G)>p\cdot\frac{(t-1)(2p-1)+1}{t(2p-1)+1}\cdot n$.	
	Let $K \subseteq V$ be the set obtained from Proposition~\ref{prop-hero}. We claim that
	\begin{equation}\label{eq-K}
		|K|\ge t+1.
	\end{equation}
	To see this, observe first that for each $y\in V$, by Proposition~\ref{prop-hero}(ii), 
	$$\sum_{x \in K}\tilde{w}(x,y)=\gamma_K(y)+p\cdot\mathbbm{1}_{\{y\in K\}}\ge p.$$
	Consequently,
	\begin{align}
		pn &\leq \sum_{x \in K,y \in V}\tilde{w}(x,y) =\sum_{x \in K}\Big(pn-\sum_{y\in V}w(x,y)\Big)\nonumber\\
		&=\sum_{x \in K}(pn-d_G(x)) \leq |K|(pn-\delta(G)) < \frac{(2p-1)pn|K|}{(2p-1)t+1} < \frac{pn|K|}{t},\label{eq:ricecake}
	\end{align}
	so $|K| \geq t+1$ as claimed.
	
	Now let $J \subseteq V$ be a set of vertices with enumeration 
	$$J = \lbrace x_1,\ldots,x_k,y_1,\ldots,y_r,z_1,\ldots,z_s\rbrace,$$
	equipped with a dominating extension $w$, such that $\lbrace x_1,\ldots,x_k\rbrace = K$, $w|_K$ has size $pk-\tilde{w}(K)$ and 
	\begin{itemize}
		\item  $w(y_i) \geq 2$ for all $i \in [r]$;
		\item  $w(z_i) \geq 1$ for all $i \in [s]$;
		\item $2r+s$ is maximal.
	\end{itemize}
	Such a pair $(J,w)$ does exist. Indeed, taking $r=s=0$, we see that $J=K$ has a dominating extension of size $pk-\tilde{w}(K)$ by Proposition~\ref{prop-hero}(i).
	
	By choice, $G[J] \in \mathcal{G}_p(pk-\tilde{w}(K)+2r+s)$.
	We shall argue that $G[J]$ is the desired subgraph in $\mathcal{G}_p(pt+2)$. Suppose otherwise, then
	$pk-\tilde{w}(K)+2r+s \leq pt+1$.
	Together with~\eqref{eq-K}, this implies
	\begin{equation}\label{double}
		(2p-1)k - 2\tilde{w}(K) + 4r+2s \leq (2p-1)t+1.
	\end{equation}
	In what follows, we write $\gamma := \gamma_K$. Define $\eta: J\rightarrow \mathbb{N}$ as follows:
	\begin{itemize}
		\item $\eta(x_i) := 2p-1-\gamma(x_i)$, for all $i \in [k]$;
		\item $\eta(y_i) := 4$, for all $i \in [r]$;
		\item $\eta(z_i) := 2$, for all $i \in [s]$.
	\end{itemize}
	Note that by Proposition~\ref{prop-hero}(ii) we have
	\begin{equation}\label{eq-etap}
	\eta(x) \geq p, \quad\text{for all }x \in K.
	\end{equation}
	Further define
	$$
	H(u) := \sum_{v \in J}\eta(v)\tilde{w}(u,v),\quad\text{for all }u \in V.
	$$
	Since $\sum_{i\in [k]}\gamma(x_i)=2\tilde{w}(K)$, we have as in~\eqref{eq:ricecake} that
	\begin{align*}
		\sum_{u \in V}H(u) &= \sum_{v \in J}\eta(v)\sum_{u \in V}\tilde{w}(u,v) \leq \Big(\sum_{v \in J}\eta(v)\Big)(pn-\delta(G))\\
		&< ((2p-1)k - 2\tilde{w}(K) + 4r+2s) \cdot \frac{p(2p-1)n}{(2p-1)t+1} \stackrel{(\ref{double})}{\leq} p(2p-1)n,
	\end{align*}
	implying that there exists a vertex $u_* \in V$ with $H(u_*) < p(2p-1)$.
	
	Suppose there is some $v \in K$ with $\tilde{w}(v,u_*) = p$.
	Then
	$$
	p(2p-1)>H(u_*)\ge \sum_{x' \in K\setminus\{v\}}\eta(x')\tilde{w}(x',u_*) + p\eta(v) \stackrel{(\ref{eq-etap})}{\ge} p\cdot\gamma_{K\setminus \{v\}}(u_*) + p(2p-1-\gamma(v)),
	$$
	implying that we in fact have $\gamma_{K\setminus\{v\}}(u_*) <  \gamma(v)$. Then Proposition~\ref{prop-hero}(iii) implies that $u_* \in K$.
	If $u_* \neq v$, then $\gamma(u_*) \geq \tilde{w}(v,u_*) = p$, contradicting Proposition~\ref{prop-hero}(ii).
	So $u_*=v$, and thus $\gamma_{K\setminus\{v\}}(u_*) = \gamma(v)$, and we have obtained a contradiction.
	Therefore,
	\begin{equation}\label{eq-K-linked}
		w(v,u_*) \geq 1,\quad\text{for all }v \in K.
	\end{equation}
	So $u_* \notin K$, since $w(u_*,u_*)=0$. Proposition~\ref{prop-hero}(ii) then implies that \begin{equation}\label{gammau*}
		\gamma(u_*) \geq p.
	\end{equation}
	Consequently, recalling that $\eta(x') \geq p$ for all $x' \in K$,
	\begin{equation}\label{eq-K-Hweight}
		\sum_{x' \in K}\eta(x')\tilde{w}(x',u_*) \geq p\gamma(u_*)\ge p^2.
	\end{equation}
	For every $y \in R := \lbrace y_1,\ldots,y_r\rbrace$, we have
	$$
	p(2p-1)>H(u_*)\ge 4\tilde{w}(y,u_*)+\sum_{x' \in K}\eta(x')\tilde{w}(x',u_*)\stackrel{\eqref{eq-K-Hweight}}{\ge} 4\tilde{w}(y,u_*)+p^2,
	$$
	whence $\tilde{w}(y,u_*)<\frac{p(p-1)}{4}<p$. So \begin{equation}\label{eq-R-linked}
		w(y,u_*) \geq 1,\quad\text{for all }y \in R.
	\end{equation}
	
	Suppose that $w(z,u_*) \geq 1$ for all $z \in S := \lbrace z_1,\ldots,z_s\rbrace$. Then $u_* \notin S$, by~\eqref{eq-K-linked} $u^* \notin K$ and by~\eqref{eq-R-linked}, $u_* \notin R$.
	So $u_* \notin J$ and moreover, setting $w(u_*)=1$ gives a dominating extension to include $J \cup \{u_*\}$.
	Thus we can add $u_*$ to $S$, contradicting the maximality of $2r+s$.
	So $S' := \{z \in S: \tilde{w}(z,u_*)=p\} \neq \varnothing$.
	Then  we have
	$$p(2p-1)>H(u_*)\ge \sum_{z \in S'}\eta(z)\tilde{w}(z,u_*)+\sum_{x' \in K}\eta(x')\tilde{w}(x',u_*)\stackrel{\eqref{eq-K-Hweight}}{\ge} 2|S'|p+p\gamma(u_*), $$
	implying together with~(\ref{gammau*}) that
	\begin{equation}\label{eq:S'}
		|S'|<\frac{2p-1-\gamma(u_*)}{2} \le \frac{p-1}{2}.
	\end{equation}
	If $p=3$, then $S'$ is empty, a contradiction.
	So from now on assume $p=4$.
	Then $|S'|=1$, i.e.~there is a unique $z_* \in S$ with $\tilde{w}(z_*,u_*)=4$, and 
	\begin{equation}\label{Sweight}
		w(z,u_*)\geq 1,\quad\text{for all }z \in S\setminus\lbrace z_*\rbrace.
	\end{equation}
	The first inequality in~(\ref{eq:S'}) implies that $\gamma(u_*) \leq 4$.
	Now~(\ref{gammau*}) implies that in fact $$\gamma(u_*)=4.$$
	Also, from
	$$
	28 > H(u_*) \geq 2\tilde{w}(z_*,u_*)+ 4\sum_{y \in R}\tilde{w}(y,u_*) +\sum_{x' \in K}\eta(x')\tilde{w}(x',u_*)\stackrel{\eqref{eq-K-Hweight}}{\ge} 8+4\sum_{y \in R}\tilde{w}(y,u_*)+16,
	$$
	we deduce that $\sum_{y \in R}\tilde{w}(y,u_*)=0$. In other words,
	\begin{equation}\label{Rweight}
		w(y,u_*) =4, \quad\text{for all }y \in R.
	\end{equation}
	
	Let $T := \lbrace x \in K: \tilde{w}(x,u_*) > 0\rbrace$. By definition, $ \sum_{x \in T}\tilde{w}(x,u_*)=\gamma(u_*) =4$, implying, together with~\eqref{eq-K-linked}, that $2 \leq |T| \leq 4$.
	If $|T| \geq 3$, then the multiset $\lbrace w(x,u_*) : x \in T\rbrace$ recording the weights from $u_*$ to $T$ is either $ \lbrace 3,3,3,3\rbrace$, or  $\lbrace 2,3,3\rbrace$. In particular, by~(\ref{Rweight}), $w(y,u_*) \geq 2$ for all $y \in K \cup R$ with equality at most once.
	Together with~(\ref{Sweight}), this implies that
	we can delete $z_*$ from $S$ and add $u_*$ to $R$ to obtain a set, $J\cup\{u_*\}\setminus\{z_*\}$, having a dominating extension with larger $2r+s$, a contradiction.
	
	Thus we may assume that $|T| = 2$.
	Let $T =: \lbrace a,b\rbrace$ and $\alpha := \tilde{w}(a,u_*)$ and $\beta := \tilde{w}(b,u_*)$.
	Then $\alpha+\beta=\gamma(u_*)=4$; and Proposition~\ref{prop-hero}(iii) implies that 
	$\beta = \gamma_{T\setminus\{a\}}(u_*) = \gamma_{K\setminus\{a\}}(u_*) \geq \gamma(a)$
	and similarly $\alpha \geq \gamma(b)$. We then arrive at the final contradiction:
	\begin{align*}
		28>H(u_*) &\geq \sum_{x' \in \{a,b,z_*\}}\eta(x')\tilde{w}(x',u_*) \ge \alpha(7-\gamma(a)) + \beta(7-\gamma(b)) + 2\cdot 4\\ &\ge \alpha(7-\beta) + \beta(7-\alpha) + 8\ge  8 + 7(\alpha+\beta)-\frac{1}{2}(\alpha+\beta)^2=28,
	\end{align*}
	completing the proof.
\end{proof}

\section{Concluding remarks}\label{sec-remarks}
In this paper, we construct complex Bollob\'as-Erd\H{o}s graphs with varying rational densities, providing, for over half of the cases, the structures predicted in Conjecture~\ref{conj}. However, in general, we show that Conjecture~\ref{conj} does not hold for infinitely many cases. Several interesting problems remain.
\begin{itemize}
	     \item We can bound the clique number of the complex Bollob\'as-Erd\H{o}s graphs in Theorem~\ref{GBE} only when $\ell\le p/2$, as the convexity of the regions (the dark ones in Figure~\ref{fig:GBE}) corresponding to~\ref{it-between}(ii) is essential for our argument. The obvious question is whether we can construct a variant with density larger than $\frac{1}{2}$ for which we can bound the clique number.
	     This would imply the existence of a graph as described in Conjecture~\ref{conj} and hence would show $\varrho_p(q) \geq \varrho_p^*(q)$ for all $p \leq q$. In particular, do we have $\varrho_3(6) = \frac{1}{3}$?
	     
	     \item We have shown that the conjectured Ramsey-Tur\'an density $\varrho_p^*(q)$ in Conjecture~\ref{conj} falls short for infinitely many cases. The smallest counterexample we have constructed is a balanced almost $3$-partite graph with density $1/4$ between parts, showing that 
	     $$\varrho_{16}(22)\ge\frac{1}{6}>\frac{5}{32}=\varrho_{16}^*(22).$$
	     We then later found almost $t$-partite counterexamples for infinitely many choices of even $t$ in Theorem~\ref{thm-2tri-exact}. As the above almost $3$-partite construction for $\varrho_{16}(22)$ differs substantially from the ones in Theorem~\ref{thm-2tri-exact}, we chose not to include its proof here.
	     
	     Now that we know when $q\le 2p+1$, the asymptotic extremal graphs need not be almost bipartite, as the next step towards understanding $\varrho_p(q)$, it would be interesting to give a characterisation of pairs $(p,q)$ with $q\le 2p+1$ such that Conjecture~\ref{conj} holds.
	     
	     \item For the upper bound, it would be nice to extend Theorem~\ref{thm-upper-main} to larger values of $p$. 
\end{itemize}



\end{document}